\documentclass[12pt,draft]{amsart}
\usepackage{amssymb,amsmath,amsthm}
\usepackage[T1]{fontenc}

\usepackage[lining]{libertine}   
\usepackage{cabin}
\usepackage[libertine]{newtxmath}

\usepackage{color}

\def\e{{\rm e}}
\def\cic{\boldsymbol}
\def\eps{\varepsilon}

\def\d{{\rm d}}
\def\dist{{\rm dist}}

\def\R {\mathbb{R}}

\def\X {{\mathcal X}}

\def\B {{B}}

\def\E {{\mathbb E}}
\def\A {{\mathcal A}}

\def\RR {{\mathcal R}}
\def\F {{\mathcal F}}

\def\M {{\mathsf M}}

\def \l {\langle}
\def \r {\rangle}

\def \and{\qquad\text{and}\qquad}

\def \no#1#2#3 {{\bf #1} (#3), #2.}
\def \eds#1#2#3 {#1, #2, #3.}
\newtheorem{proposition}{Proposition}
\newtheorem{theorem}{Theorem}
\newtheorem*{theorem*}{Theorem}
\newtheorem*{proposition*}{Proposition}
\newtheorem{corollary}{Corollary}
\newtheorem{lemma}[proposition]{Lemma}
\theoremstyle{definition}

\newtheorem{remark}[proposition]{Remark}

\numberwithin{proposition}{section}
\numberwithin{equation}{section}

\title[Banach-valued  singular integrals]{Banach-valued multilinear singular integrals}

 \author{Francesco Di Plinio} \address{\noindent Department of Mathematics, University of Virginia,  \newline \indent Kerchof Hall,  Box 400137, Charlottesville, VA 22904-4137, USA   }
 \email[F.\ Di Plinio]{francesco.diplinio@virginia.edu}
 \author{Yumeng Ou}
 \address{\noindent Department of Mathematics, Massachusetts Institute of Technology, \newline \indent  77 Massachusetts Avenue, Cambridge, MA 02139, USA }
\email[Y.\ Ou]{yumengou@mit.edu}
\thanks{F Di Plinio was partially
supported by the National Science Foundation under the grant
   NSF-DMS-1500449 and NSF-DMS-1650810. Y Ou was partially supported by   the National Science Foundation under the grant
  NSF-DMS 0901139}

\begin{document}

\begin{abstract} We develop a general framework for the analysis of operator-valued multilinear multipliers  acting on  Banach-valued functions. 
Our main result is a  Coifman-Meyer type theorem  for operator-valued multilinear multipliers acting on suitable tuples of UMD spaces. A concrete case of our theorem is a multilinear generalization of Weis's operator-valued H\"ormander-Mihlin linear multiplier theorem \cite{Weis01}.  Furthermore,   we  derive from our main result a wide range of mixed $L^p$-norm estimates for multi-parameter multilinear paraproducts,    leading to a novel  mixed norm version of the partial fractional  Leibniz rules of  Muscalu et.\ al.\ \cite{MPTT1}. Our approach works just as well for the more singular tensor products of a one-parameter Coifman-Meyer multiplier with a bilinear Hilbert transform, extending results of Silva \cite{Silva12}. 
  We also prove several operator-valued $T (1)$-type theorems both in one parameter, and of multi-para\-meter, mixed-norm  type. A distinguishing feature of our $T(1)$ theorems is that   the usual  explicit assumptions on the distributional kernel of $T$ are replaced with testing-type conditions. Our proofs rely on a newly developed Banach-valued version of the outer $L^p$ space theory of Do and Thiele \cite{DoThiele15}.
\end{abstract}
 
\maketitle



\section{Introduction}

The first main purpose  of this article is a systematic \emph{operator-valued} generalization of the    theory of  \emph{multilinear} Calder\'on-Zygmund singular integrals, pioneered by  Coifman and Meyer in \cite{CM1}, Kenig and Stein in \cite{KS}  and Grafakos and Torres in \cite{GT}. Besides the intrinsic interest,   operator-valued multilinear multiplier theorems  provide a novel approach to  multilinear  \emph{multi-parameter} singular integrals, such as  the multi-parameter paraproducts of Muscalu, Pipher, Tao, Thiele \cite{MPTT1}, which will be tackled by suitable iterations of the one-parameter, operator-valued  results. This point of view has been fruitfully employed in the linear theory. For instance, Journ\'e's multi-parameter $T(1)$ theorem \cite{Jou} is obtained by iteration of the one-parameter, vector-valued  result. However, this idea seems to be   new in the multilinear case.
As a byproduct of our methods,  we  obtain a novel  mixed norm version of the multi-parameter fractional  Leibniz rules of \cite{MPTT1}, answering a question posed by Kenig in connection with the work \cite{KPV93} on local well-posedness of the generalized Korteweg-de Vries equation.

\begin{corollary}\label{LR}[Fractional Leibniz rules]
For any $\alpha, \beta>0$, and functions $f_i:\R_x\times\R_y\to \mathbb C$, $i=1,2$, there holds
\[
\begin{split}
\left\|D_x^\alpha D_y^\beta (f_1f_2)\right\|_{L_x^{r,\infty}(L_y^s)}\lesssim &\quad \left\|f_1\right\|_{L_x^{p_1}(L_y^{q_1})}\left\|D_x^\alpha D_y^\beta f_2\right\|_{L_x^{p_2}(L_y^{q_2})}+\left\|D_x^\alpha f_1\right\|_{L_x^{p_1}(L_y^{q_1})}\left\|D_y^\beta f_2\right\|_{L_x^{p_2}(L_y^{q_2})}\\
&+\left\|D_y^\beta f_1\right\|_{L_x^{p_1}(L_y^{q_1})}\left\|D_x^\alpha f_2\right\|_{L_x^{p_2}(L_y^{q_2})}+\left\|D_x^\alpha D_y^\beta f_1\right\|_{L_x^{p_1}(L_y^{q_1})}\|f_2\|_{L_x^{p_2}(L_y^{q_2})}
\end{split}
\]
whenever 
\[\textstyle
\frac{1}{q_1}+\frac{1}{q_2}=\frac{1}{s},\qquad 1<q_1, q_2\leq\infty,\;1<s<\infty
\]
and
\[\textstyle
\frac{1}{p_1}+\frac{1}{p_2}=\frac{1}{r},\qquad 1\leq p_1, p_2\leq \infty,\;\max\left(\frac{1}{2},\frac{1}{1+\alpha}\right)\leq r<\infty.
\]
And the symmetric estimate holds true with the two variables interchanged. If $r>\max\left(\frac{1}{2},\frac{1}{1+\alpha}\right)$, the strong type estimate also holds.
\end{corollary}
While the recent article \cite{TorWard15} by Torres and Ward   obtains the mixed norm estimates of Corollary \ref{LR}  in the reflexive Banach  range  through techniques of different nature, our method allows for choices of  exponents $(p_1,p_2,r)$  outside the Banach triangle range, up to the natural weak endpoint $p_1=p_2=1$. In fact, the need for an operator-valued \emph{multilinear} theory arises when dealing with exponent tuples outside the Banach triangle range, which cannot be tackled by simply combining estimates for  suitable \emph{linear} multiparameter square-maximal functions and H\"older's inequality. See for example the big difficulty gap between the Banach and outside-Banach ranges in the multiparameter theorem of \cite{MPTT1}.

The second main goal we pursue in our article are operator-valued  $T(1)$ theorems, in one and multiple parameters,  where the more customary kernel assumptions are completely replaced by conditions of testing type. In the operator-valued setting, such multi-parameter $T(1)$ theorems seem to be the first of their kind. Even in the scalar-valued case, our result has the advantage that in however many parameters, the testing condition is always of simple tensor product form, allowing for the argument to be iterated.

Both families of results introduced above are obtained within the newly developed setting of  Banach-valued   outer $L^p$ theory, generalizing the scalar case introduced by Do and Thiele in \cite{DoThiele15}.
   In contrast to  the related Banach-valued tent spaces of \cite{HytNeePor08,HytWeis10,KempVV}, the local nature of our setting allows for estimation of  multilinear multipliers  outside the Banach triangle, as in, for instance, \cite{MPTT1}. Another advantage is that suitable modifications of our tools are amenable to treat  modulation-invariant operators, in a similar fashion to Section 5 of \cite{DoThiele15}. This will be performed in   a forthcoming article.

\subsection{An operator-valued Coifman-Meyer theorem}
Our first main result is an operator-valued (or more precisely, trilinear form-valued) theorem of Coifman-Meyer type, which is Theorem \ref{CMT}, detailed in Section \ref{SecCMT}. We briefly sketch its content here, without the pretense of being fully rigorous. The basic setup  features  three complex UMD spaces $\X_1,\X_2,\mathcal X_3$  and a trilinear form $\mathsf{id}: \X_1\times \X_2 \times \X_3\to \mathbb C$ such that   \begin{equation}
\label{trcontfm}
|\mathsf{id}(x_1,x_2,x_3)|\leq \prod_{j=1,2,3}\|x_j\|_{\X_j}, \qquad x_j \in \X_j, \,j=1,2,3.\end{equation}
Thus, $\X_3$ can  be identified with a subspace of $B(\X_1,\X_2)$, the space of bilinear continuous forms,  with the $\mathcal X_3$-norm dominating the operator norm, and analogous identification can be performed  for each permutation of $\X_j$.    The class of UMD Banach spaces, characterized by the property that singular integrals, in particular the Hilbert transform, admit a $L^p(\X)$-bounded extension, is the natural setting for   Banach-valued \emph{linear} singular integrals: we refer to  the seminal works of Burkholder \cite{Burk1983} and Bourgain \cite{Bou1983,Bou1986}, and the more recent articles \cite{HytWeis06,HytLac2013,Weis01}.

Under suitable geometric assumptions (referred to as \emph{nontangential RMF property}) on the triple $\X_j$,  which we describe and motivate at the end of the introduction, we prove the full range  of   $L^p(\X_j)$ estimates for trilinear forms of the type
\begin{equation} \label{introform}
\Lambda_m(f_1,f_2,f_3) =\int \displaylimits_{\xi_1+\xi_2+\xi_3=0} m(\xi_1,\xi_2,\xi_3) \big[ \widehat f_1(\xi_1),\widehat f_2(\xi_2), \widehat f_3(\xi_3)\big]\,\d\xi
\end{equation}
 initially defined for triples of smooth, $\X_j$-valued functions $f_j$. Here $m$  is a smooth away from the origin, \emph{trilinear form-valued} symbol satisfying a   version of the Coifman-Meyer smoothness condition based upon the newly introduced notion of $  \mathcal{R}$-\emph{boundedness}  for  families of \emph{multilinear} operators.  
 This concept is an important novelty of the present article, and it appears to be the correct multilinear generalization of the  $\mathcal{R}$-boundedness of families of \emph{linear} operators, which is  the right substitute for uniform boundedness when working with operator-valued \emph{linear} multipliers, see for instance \cite{Weis01}. 

Our randomized condition is obviously satisfied in the simple case $m=\tilde m \mathsf{id}$ for some scalar Coif\-man-Meyer symbol $\tilde m$.
Furthermore, we detail its application to two seemingly unrelated particular cases.
The first is the following bilinear generalization of Weis's theorem from \cite{Weis01}, involving the same $\mathcal{R}$-bounded H\"ormander-Mihlin condition on the operator-valued symbol. In particular, the corollary below applies to the case of $\X$ being the Schatten-von Neumann class $S_p=L^p(\mathcal B(H) , \mathrm{tr})$, see  \cite{PX} for definition and basic properties. To the best of our knowledge, Corollary \ref{introthm1} is the first multilinear multiplier theorem involving noncommutative $L^p$ spaces.

\begin{corollary} \label{introthm1} 
 Let $\X$ be a complex UMD Banach space with the nontangential RMF property and denote by $\mathcal L(\X)$ the space of linear endomorphisms of $\mathcal X$. Let $
 \mathsf{m}:\R^2\to  \mathcal L(\X)$, smooth away from the origin and such that the family of operators
 $$
 \mathcal M=
 \left\{|\xi|^\alpha \partial_{\xi}^\alpha \mathsf{m}(\xi)[\cdot]: \xi \in \R^2\backslash \{0\}, |\alpha|\leq N\right\}$$ is $\mathcal{R}$-bounded in $\mathcal L (\X)$, in the sense of condition \eqref{Rboundsop}.
Then the operator $$
T_{\mathsf m}(f_1,f_2)(x)= \int_{\R^2}\mathsf m(\xi_1,\xi_2) [\widehat f_1(\xi_1)] \widehat f_2 (\xi_2) \e^{-ix(\xi_1+\xi_2)} \,\d\xi_1\d \xi_2
 $$
 initially defined for Schwartz, $\X$-valued functions $f_1$ and Schwartz functions $f_2$, extends to a  bounded bilinear operator
\begin{align*}
&T_{\mathsf m}: L^{p_1}(\R; \X ) \times L^{p_2}(\R ) \to L^{r}(\R; \X ),   \qquad \textstyle \frac{1}{p_1} + \frac{1}{p_2} = \frac 1r, \; 1<p_1,p_2 \leq \infty, \, r<\infty;\\ &T_{\mathsf m}: L^{1}(\R; \X ) \times L^{1}(\R ) \to L^{\frac12,\infty}(\R; \X),
\end{align*}
with operator norms depending only on the $\mathcal R$-bound of $\mathcal M$, and on the   exponents $p_1,p_2$.
\end{corollary}

The second particular case  of our main Theorem \ref{CMT} is its restatement for three UMD function lattices  $\X_j$, say, $\X_j=L^{q_j}$ on the same measure space, tied in \eqref{trcontfm} by H\"older's inequality. In this setting, our multilinear $\RR$-boundedness condition on the symbol $m$ has a very transparent interpretation: see \eqref{condLqj} of the next corollary. 

\begin{corollary} \label{REVCOR}
Let $(\mathcal N,\nu)$ be a $\sigma$-finite measure space, and
$
 \textstyle \frac{1}{q_1} + \frac{1}{q_2} + \frac{1}{q_3} =1,$ $1<q_1,q_2,q_3<\infty,
$ be a given H\"older triple.
Let $$
\xi\in \R^3: \xi_1+\xi_2+\xi_3=0 \mapsto m(\xi)=m(\xi)[g_1,g_2,g_3], \qquad g_j \in L^{q_j}(\mathcal N,\nu),\, j=1,2,3, 
$$ be a trilinear form-valued function, smooth away from the origin. Setting
$$
\mathcal M=
 \left\{|\xi|^\alpha \partial_{\xi}^\alpha {m}(\xi)[\cdot,\cdot, \cdot ]:  \xi_1+\xi_2+\xi_3=0, |\alpha|\leq N\right\},$$
 assume that
\begin{equation} \label{condLqj}\begin{split}
&\quad\sum_{j=1}^n \mathsf{m}_j [g_1^j, g_2^j, g_3^j]\\ & \lesssim C_{\mathcal M} \Big\|\Big(\sum_{j=1}^n |g_{\sigma(1)}^j|^2\Big)^{1/2}\Big\|_{L^{q_1}(\nu)}\Big\|\Big(\sum_{j=1}^n |g_{\sigma(2)}^j|^2\Big)^{1/2}\Big\|_{L^{q_2}(\nu)} \Big\|\sup_{1\leq j\leq n} |g_{\sigma(3)}^j| \Big\|_{L^{q_3}(\nu)}
\end{split}
\end{equation}
uniformly over integers $n$, over choices $\{\mathsf{m}_j :j=1,
\ldots n\},\,\mathsf m_j \in   \mathcal M$, and over permutations $\sigma $ of $\{1,2,3\}$. Then, the dual operators $T_m$ to the corresponding form $\Lambda_m$ of \eqref{introform} extend to bounded bilinear operators
\begin{align*}
&  T_{  m}: L^{p_1}(\R; L^{q_1}(\nu)) \times L^{p_2}(\R ; L^{q_2}(\nu)) \to L^{r}(\R; L^{q_3'}(\nu) ),\\   & \textstyle \frac{1}{p_1} + \frac{1}{p_2} = \frac 1r, \; 1<p_1,p_2 \leq \infty, \, r<\infty;\\ & T_{  m}: L^{1}(\R; L^{q_1}(\nu)) \times L^{1}(\R ; L^{q_2}(\nu)) \to L^{\frac12,\infty}(\R; L^{q_3'}(\nu) )
\end{align*}
with operator norm depending only on $C_{\mathcal M}$ and on the   exponents $p_1,p_2$.

\label{introthm2}
\end{corollary}

Bounds of the  type \eqref{condLqj} are quite familiar in multilinear harmonic analysis. They are classically known to hold for families $\mathcal M$ of (uniformly) Coifman-Meyer symbols, as well as for families $\mathcal M$ of bilinear Hilbert transforms \cite{LT1, LT2}, as shown by Silva in \cite{Silva12}.
 In \cite{Silva12}, Silva proved $L^p$ estimates  for a tensor product of  a bilinear Hilbert transform with  a Coifman-Meyer multiplier, positively answering a question by Muscalu et.\ al.\ in \cite{MPTT1}, where analogous results for the less singular bi-parameter  Coifman-Meyer multipliers had been obtained. Applying Corollary \ref{introthm2}, we can extend the results of both \cite{MPTT1} and \cite{Silva12}  to mixed-norm spaces, when the inner exponents lie in the Banach triangle.  

 \begin{corollary} \label{BHT}
 Let $m=m(\xi,\eta)$ be a complex function defined on  $$\{(\xi,\eta)\in \R^3\times \R^3: \xi_1+\xi_2+\xi_3=\eta_1+\eta_2+\eta_3=0\},$$  smooth away from the origin and satisfying
 \begin{equation} \label{mpest}
 \sup_{\substack{\xi_1+\xi_2+\xi_3=0\\ \eta_1+\eta_2+\eta_3=0} 
 }
 \dist(\xi,\Gamma_0)^{|\alpha|}  \dist(\eta,\Gamma_1)^{|\beta|} \big|\partial_{\xi}^\alpha \partial^\beta_\eta m(\xi,\eta)\big| \leq K, \qquad |\alpha|,|\beta| \leq N\end{equation}
 where $\Gamma_0=\{0\}$, $\Gamma_1$ is a non-degenerate  proper subspace of $\{\eta_1+\eta_2+\eta_3=0\}$. Then, the dual operators $T$ to the trilinear form  
 $$
 \Lambda_{m} (f_1,f_2,f_3) =\int\displaylimits_{\substack{\xi_1+\xi_2+\xi_3=0\\ \eta_1+\eta_2+\eta_3=0} 
 } m(\xi, \eta) \widehat f_1(\xi_1,\eta_1)\widehat f_2(\xi_2,\eta_2)\widehat f_3(\xi_3,\eta_3)\, \d \xi \d \eta$$
initially defined for $f_j\in \mathcal S (\R^2)$, extend to  bounded linear operators
\begin{align*}
&T: L^{p_1}(\R; L^{q_1}(\R)) \times  L^{p_2}(\R; L^{q_2}(\R)) \to  L^{r}(\R; L^{q_3'}(\R)),  
\\ & \textstyle \frac{1}{p_1} + \frac{1}{p_2} = \frac 1r, \; 1<p_1,p_2 \leq \infty, \, r<\infty; \\& T: L^{1}(\R; L^{q_1}(\R)) \times  L^{1}(\R; L^{q_2}(\R)) \to  L^{\frac12,\infty} (\R; L^{q_3'}(\R)),
\end{align*}
for all tuples
$
  1<q_1,q_2,q_3< \infty, \frac{1}{q_1}+ \frac{1}{q_2}+\frac{1}{q_3}=1,$ with operator norm depending on the exponent tuples and on the constant $K$ appearing in \eqref{mpest}.
\end{corollary}

Note that the proof of \cite{MPTT1} relies on heavy product-BMO theory and Journe's lemma-type arguments. The article \cite{MPTT06}, by the same authors, extends and simplifies the approach of \cite{MPTT1}. The product theory content of \cite{MPTT06} and   \cite{Silva12} is essentially restricted to the $L^p$ boundedness of the double square function and of the strong maximal function. Our approach to Corollary \ref{BHT}, relying on the iteration of vector valued, multilinear type bounds, completely avoids product theory, thus allowing  for the extremal tuple $(1,1,1/2)$ in one of the variables. Moreover, the $\Gamma_1=\{0\}$ case of Corollary \ref{BHT} implies via standard arguments, see for instance \cite{MPTT1,MuscSchlII}, the aforementioned Leibniz rules Corollary \ref{LR}.


\subsection{One and multi-parameter $T(1)$-theorems}
A second significant part of the  article is devoted to deducing several Banach-valued $T(1)$ theorems, in one and multiple parameters, within the same outer measure theory setting. The strength of $T(1)$ theorems lies in their power of characterizing boundedness of operators using a simple set of conditions. In the Banach-valued setting, perhaps the most general such type of result is the   nonhomogeneous, local $T(b)$ theorem from \cite{Hyt2014}. However, in the global, Euclidean $T(1)$ case, the same generality is reached by \cite[Corollary 1.6]{HytHann14}, to which we refer for comparison. Both the  assumptions and the proofs  of the modern  Banach-valued $T(1)$ theorems are by now rather customary in form.  Restricting for simplicity  to the paraproduct free case, the former involve a randomized version of the standard kernel assumptions and of the weak boundedness property, i.e.\ testing over indicator functions of intervals. The latter are usually based on some realization of Hyt\"onen's  dyadic representation theorem from \cite{HytA2}, of which  \cite{HytHann14} provides an operator-valued version.

In Section \ref{T1secstat} of this article, we take on a different approach to operator-valued $T(1)$-theorems, completely replacing  the (randomized versions of the) kernel assumptions with conditions of testing type. Sacrificing a bit of rigor, we now provide an example of our results. We start from a bilinear form $\Lambda(f_1,f_2)$ initially defined on pairs of $\X_j$-valued Schwartz functions $f_j$.
Fixing a mean zero Schwartz function $\phi$ which is good for Calder\'on's reproducing formula, we write $\phi_{x,s}$ for the $L^1$-normalized $s$-dilation of $\phi$ centered at $x\in \R$ and  construct the family of bilinear forms
\begin{equation} 
\label{introeqA}\begin{split} &
(\xi,\eta) \in \X_1\times \X_2\mapsto A(x,s,y,t) \Lambda(\phi_{x,s}\xi,\phi_{y,t}\eta), \\ & A(x,s,y,t):=\frac{\max(s,t,|x-y|)^2}{\min(s,t)}. \end{split}
\end{equation}
Our conclusion, in Theorem \ref{onepara}, is that $\Lambda$ extends to a bounded bilinear form on $L^p(\X_1)\times L^{p'}(\X_2)$ if the above family of bilinear forms is randomized bounded, in the sense specified in  Section \ref{SecRbound}. This should be thought of as a $T(1)=0$-type theorem, and in fact, the randomized testing condition of our result can be obtained from the standard $T(1)$ assumptions involving the kernel of the dual operator to $\Lambda$ in the paraproduct free case.  

 Theorem \ref{onepara}, albeit restricted to  the less general Euclidean setting,  has the advantage of not requiring any \emph{a-priori} assumption on the distributional kernel of $\Lambda$ (or $T$). Besides the wider formal generality, we believe that   checking testing-type condition like \eqref{newwbp} rather than kernel assumptions is closer to being computationally feasible and thus more realistic. 
  We also remark that our methods do not involve dyadic representation theorems of any sort, unlike, for instance, the approach of \cite{HytHann14} and references therein. This can also be regarded as computationally advantageous: the trade-off of the extremely neat probabilistic expression of the dyadic representation theorem of \cite{HytA2} is the rather large parameter space involved.

A further nice feature of Theorem \ref{onepara} is that, under the additional necessary assumption of   Pisier's property $(\alpha)$ on  $\X_1,\X_2$, it can be iterated to obtain a mixed norm, multi-parameter operator-valued $T(1)=0$ theorem, Theorem \ref{bipara}. To the best of our knowledge, this is the first $T(1)$ type result in the multi-parameter operator-valued setting. We also formulate  $T(1)\in$ BMO type results in the same spirit, see Theorems \ref{FullT1thm} and \ref{biparafull}. 

In particular, Theorem \ref{biparafull} recovers  the multi-parameter, scalar valued $T(1)$-theo\-rems of \cite{PottVilla11}, replacing the \emph{a-priori} full and mixed kernel assumptions with testing-type conditions. Even  when compared with other multi-parameter $T(1)$ type theorems available in the scalar case, our result has the advantage that in however many parameters, the testing condition is always of   simple tensor product form, see Theorem \ref{bipara}. In contrast, for instance, in the formulation of \cite{PottVilla11}, the total number of mixed type $T(1)$ conditions that are needed grows exponentially with   the number of parameters. 
To give the readers a better idea of how the testing condition is iterated, we state here the following scalar-valued bi-parameter $T(1)$ theorem, which is a special case of Theorem \ref{bipara}.
\begin{corollary}
Let $\Lambda:\mathcal{S}(\R)\otimes\mathcal{S}(\R)\times\mathcal{S}(\R)\otimes\mathcal{S}(\R)\rightarrow\mathbb{C}$ be a bilinear form and $\phi $ be a   function as above. Assume that, with $A(\cdot,\ldots,\cdot)$ as in \eqref{introeqA},
\[
\sup_{\substack{x_1,x_2,y_1,y_2\in\R\\s_1,s_2,t_1,t_2\in\R_+}}\Big|A(x_1,s_1,y_1,t_1) A(x_2,s_2,y_2,t_2)\Lambda(\phi_{x_1,s_1}\otimes\phi_{x_2,s_2},\phi_{y_1,t_1}\otimes\phi_{y_2,t_2})\Big|<\infty.
\]
Then, 
\[
|\Lambda(f,g)|\leq C_{p,q}\|f\|_{L^p(\R;L^q(\R))}\|g\|_{L^{p'}(\R;L^{q'}(\R))} \qquad \forall 1<p,q<\infty.
\]
\end{corollary}

\subsection*{The UMD and RMF assumptions} \label{Ssintroass} Before ending the introduction, we return to our UMD and RMF assumptions on the triple of spaces $\X_j$ and the identification \eqref{trcontfm} in the main result, Theorem \ref{CMT}. To motivate them, we momentarily break the symmetry, which can be  restored by applying the reasoning below to all permutations of $\X_j$, and take     $f_3$  to be a $\X_3$-valued function with frequency support near the origin. Then, by frequency support considerations,
 $$
 (f_1,f_2) \mapsto \Lambda_m(f_1,f_2,f_3) 
 $$
 is essentially a bilinear singular integral form, and it is thus necessary to assume that $\X_1,\X_2$  are UMD spaces in order to have $L^p$-bounds of any type.  To tackle the $\X_3$ component, it is natural to look for some control on the maximal averages of the $\X_3$-valued function  identified by \eqref{trcontfm} with a \emph{bilinear form-valued function}. As usual in Banach-valued theory, this control must be on randomized bounds rather than uniform operator bounds, see Section \ref{SecRbound}. We formulate this requirement in terms of a geometric property of $\X_3$ and $\mathsf{id}$, independent of UMD, which we have referred to as \emph{nontangential Rademacher maximal function} property, RMF for short.  
 
 Our RMF property is closely related to the dyadic RMF property first introduced in  \cite{HytMcPor08} and further studied in \cite{Kemp11,Kemp13}. A detailed investigation of the nontangential RMF, and a comparison with the dyadic version, is given in Section \ref{SecPP}. Here, we mention again that our setting is general enough to include both settings of  UMD lattices and of noncommutative $L^p$ spaces. The RMF property in the latter case is proved by relying upon a noncommutative Doob's inequality due to Mei \cite{Mei07}.

\subsection*{Structure} Section \ref{SecRbound} contains   our formulation  of the notion of $\mathcal R$-boundedness for families of bilinear and trilinear forms, which will be an integral part of the assumptions in our main results. In Section 3, we introduce the operator-valued model paraproduct forms, whose   $L^p$-bounds are stated in Proposition \ref{thmmodel}, and proved in Section \ref{Secpfthmmodel}. In the same section, we also present the notion of nontangential RMF and establish its scope: the proofs of the lemmata are postponed to Section \ref{ssprfrmf}. Section \ref{SecCMT} contains the operator-valued Coifman-Meyer theorem, which is reduced to Proposition \ref{thmmodel} in Section \ref{Secdec}, and the derivation of Corollaries \ref{introthm1} to \ref{BHT}. The $T(1)$ results are formulated in Section \ref{T1secstat} and proved in Section \ref{SecT1pf}. In Section \ref{SecRM}, we introduce the randomized norms whose concrete versions will be used in the construction of the outer $L^p$-spaces, which is in turn performed in Section \ref{SecOM}. In Sections \ref{SecLCET} and \ref{SecRMF} we state and prove the vector-valued Carleson embedding theorems.   
\subsection*{Acknowledgements}
The authors are grateful to Michael Lacey, Jill Pipher, Christoph Thiele and Sergei Treil for  providing useful feedback on an early version of the manuscript.
 The authors also thank   Mikko Kemppainen for his feedback on the relationship between dyadic and nontangential Rademacher maximal function.  Finally, the authors want to express their gratitude to the anonymous referee for the valuable suggestions which considerably improved the exposition. 

\subsection*{Note added in proof}
After the first version of this article was completed in June 2015,   vector valued estimates and mixed-norm bounds for  multiparameter singular integrals, in the same spirit of our Corollaries \ref{REVCOR} and \ref{BHT},  have been obtained in the preprint \cite{BenMusc15}, which has since appeared in print in \cite{BMfinal}.  Interestingly, the results of \cite{BenMusc15,BMfinal} are obtained by completely different methods not involving UMD spaces.

\section{$\mathcal{R}$-boundedness}  \label{SecRbound}
This section contains  our formulation  of the notions of $\mathcal R$-boundedness for families of operators between Banach spaces. In the context of \emph{linear} operators, this is a well established concept, see \cite{KunstWeis04,Weis01} and the numerous references therein. Due to the additional need for symmetry in some of the arguments concerning our multilinear extension of this notion, we prefer to work with bilinear or trilinear forms instead.

 Without further mention, we consider Banach spaces over the complex scalar field and we denote by $B(\X_1,\X_2)$ the Banach space of complex valued bilinear forms on the Banach spaces $\X_1,\X_2$. A similar notation is adopted for    trilinear forms.
We will represent as  $\{r_n:n \in \mathbb N\}$   the standard Rademacher sequence on the probability space $[0,1)$ with Lebesgue measure.  When no confusion arises about which probability space is being considered, we use the notation $\mathbb E$ for expected value.
For a Banach space $\mathcal X$, we define $ \mathrm{Rad}(\X)$ to be the closure of $ \mathrm{span}\{r_jx_j: x_j\in \X, j=1,2,\ldots,\}\subset L^2([0,1), 
\X)$ with respect to the norm
$$
\Big\|\sum_{j=1}^N r_j x_j\Big\|_{ \mathrm{Rad}(\X)} := \left(\ \E_\omega \Big\| \sum_{j=1}^N r_j(\omega)   x_j \Big\|_{\X}^2 
 \right)^{\frac12}.
$$
Notice that, by Kahane-Khintchine inequality for randomized sums, we can replace the exponent $2$ in the above definition by any $0<p<\infty$ and obtain an equivalent (quasi)-norm.
\subsection{$\mathcal R$-boundedness of bilinear forms}
Let $\X_1, \X_2$ be Banach spaces and $Z$, for now, an index set. We define the $\mathcal R$-bound of a  family of bilinear forms $\{\Lambda_{z}:z \in Z\}\subset B(\X_1, \X_2)$  as\begin{equation}
\label{rbound-ab}\begin{split}
&\quad \mathcal R_{\X_1,\X_2}( \{\Lambda_z \}) \\& := \textrm{the least }\, C>0\, \textrm{ s.t.}\;\Big| \sum_{j=1}^n \Lambda_{z_j}(x_1^j,x_2^j) \Big|\leq C \prod_{k=1,2}\Big( \E_\omega \Big\| \sum_{j=1}^N r_j(\omega)   x^j_k \Big\|_{\mathcal X_k}^2\Big)^{\frac12}  \end{split} \end{equation}
for all $N$ and all choices $z_j \in Z, x^k_j \in \X_k,$ $k=1,2$, $j=1,\ldots, N$. 
In other words, defining the family of bilinear forms indexed by $\cic z \in Z^N$, $N \in \mathbb N$
$$
\cic z \in Z^N \mapsto \Lambda_{\cic z} \in B(\mathrm{Rad} (\X_1) ,\mathrm{Rad} (\mathcal X_2)),\qquad  \Lambda_{\cic z}\Big(\sum_{j=1}^N r_j    x_1^j,  \sum_{j=1}^N r_j   x_2^j \Big) := \sum_{j=1}^N \Lambda_{z_j}(x_1^j,x_2^j)
$$
we have
\begin{equation}
\label{rbound-def}
\mathcal R_{\X_1,\X_2}( \{\Lambda_z: z \in Z \}) := \sup_{N \in \mathbb N} \sup_{\cic z \in Z^N } \| \Lambda_{\cic z}\|_{B(\mathrm{Rad} (\X_1) , \mathrm{Rad} (\mathcal X_2))}.
\end{equation}

An additional geometric property of the Banach spaces $\X_1,\X_2$ is often needed when dealing with iterated randomization arguments. 
  We say that a Banach space $\X$ has Pisier's property $(\alpha)$ if
$$
\E_\omega \E_{\omega'} \Big\| \sum_{j,k=1}^N r_j(\omega) r_k(\omega') \alpha_{jk} x^{jk}\Big\|_{\X} \leq C  \E_\omega \E_{\omega'} \Big\| \sum_{j,k=1}^N r_j(\omega) r_k(\omega')   x^{jk}\Big\|_{\X} 
$$
for all integers $N$, all choices of unimodular scalars $\alpha_{jk}$ and vectors $x^{jk}$. It turns out (\cite{Clem00,HytWeis08}) that  $(\alpha)$ is equivalent to the following property.
\begin{lemma}\label{lemmaalpha} Let $\X_1, \X_2$  be  Banach spaces with property $(\alpha)$. Then for all families $\{\Lambda_z: z\in Z\}\subset B(\X_1,\X_2)$,
$$
\mathcal{R}_{\mathrm{Rad}(\X_1),\mathrm{Rad}(\X_2)} \big( \left\{ \Lambda_{\cic z}: \cic z\in Z^N , N \in \mathbb N\right\}  \big) = \mathcal{R}_{ \X_1,\X_2} \big( \left\{ \Lambda_{ z}: z\in Z\right\}  \big).
$$ 
\end{lemma}

\label{ss11}
We will usually consider the $\mathcal R$-bound of a family of vectors from $\X_j$, $j=1,2,3$ when the three complex Banach spaces are identified by \eqref{trcontfm}.  The $\mathcal R$-bound of a family $ \cic{x}_3 \subset \mathcal 
\X_3$ is then naturally defined as\begin{equation}
\label{rbound}
\mathcal R_{\X_1,\X_2}( \cic{x}_3 ) :=  \mathcal R_{\X_1,\X_2}\big(  \{(x_1,x_2) \mapsto \mathsf{id}(x_1,x_2, x_3):  x_3 \in  \cic{x}_3   \}\big).
\end{equation}
 \begin{remark} \label{remtype}
By using orthogonality of the Rademachers, we learn that, if $\X_1$ and $\mathcal X_2$ are both Hilbert spaces, then $\mathcal R(\cic{x}_3)= \sup_{x_3 \in \cic{x}_3} \|x_3\|_{\mathcal B (\X_1, \mathcal X_2)}$. More generally, it is a result of Pisier, appearing in \cite{ArendtBu02}, that  $\mathcal R(\cic{x}_3)$ is comparable to the supremum of the operator norms  if and only if  $\X_1$, $\X_2$ both have cotype 2.   \end{remark}

\subsection{$\mathcal R$-boundedness of trilinear forms}
The most natural way to formulate our H\"or\-man\-der-Mihlin assumption  on operator-valued multilinear multiplier forms   involves a further notion of $\mathcal{R}$-boundedness for families of   trilinear continuous forms in $\B( \X_1,\X_2,\X_3) $, which we now introduce.  The definition below is actually a generalization of $\ell^2-\ell^2-\ell^\infty$-valued H\"older-type bounds for families of  forms defined on a triple of UMD lattices. 

 For a family $\{\Lambda_z:z \in Z\}$ of multilinear forms in $\B( \X_1,\X_2,\mathcal X_3) $, we define  
\begin{equation} \label{Rboundsform} \begin{split}  \RR_{\mathcal X_3|\X_1,\X_2}^{}(\{\Lambda_z\}) & := \textrm{the least }\, C>0\, \textrm{ s.t.} \sum_{j=1}^N {\Lambda}_{z_j}(x_1^j,x^j_2, x_3^j) \\ &
\leq C \RR_{\X_1,\X_2}\left(\{x_3^1,\ldots,x_3^N\} \right) \prod_{k=1,2}\left(\E_\omega \Big\| \sum_{j=1}^N r_j(\omega)   x_k^j \Big\|_{\X_k}^2 \right)^{\frac12}  \end{split} 
\end{equation}
for all $N$ and all choices $z_j$,  $x^j_k\in \X_k$, $k=1,2,3$.
In other words, \begin{equation} 
\label{Rlambda}\begin{split}&\quad 
\RR_{\X_3|\X_1,\X_2}(\{\Lambda_{z}\})\\ &=   {\mathcal R_{\X_1,\X_2}\left(\big\{(x_1,x_2)\mapsto{\Lambda}_{z}(x_1,x_2,x_3): x_3 \in \cic{x}_3,\, \mathcal R_{\X_1,\X_2}(\{\cic{x}_3 \})= 1,\, z
\in Z  \big\} \right).}\end{split}
\end{equation}
Obviously $\RR_{\mathcal X_3|\X_1,\X_2}(\{\lambda\mathsf{id}: |\lambda|\leq 1\})=1$. When no confusion arises, we agree to omit the subscripts in \eqref{Rlambda}. Finally, we set  
\begin{equation}
\label{Rlambda2} \mathcal R_{ \X_1,\X_2,\X_3}(\{\Lambda_{z}:z\in Z\}):=
\sup_{\sigma} \mathcal R_{ \X_{\sigma(3) }|\X_{\sigma(1)},\X_{\sigma(2)}}(\{\Lambda_{z}:z \in Z\}),
\end{equation}
with  supremum being taken over all cyclic permutations $\sigma $ of $\{1,2,3\}$. In the attempt of clarifying   the meaning of condition \eqref{Rlambda},   we  pause our exposition to provide some examples of identifications \eqref{trcontfm} falling within the scope of our theory. 

\subsubsection{The case $\X_1=\X,\X_2=\X',\X_3=\mathbb C$} In this case, the identification    \eqref{trcontfm} is  the obvious one. Let  $\{T_{z}:z \in Z\}\subset \mathcal L(\X)$ be an $\mathcal R$-bounded family of operators with $\mathcal R$-bound $K>0$, that is 
 \begin{equation} \label{Rboundsop}
 \ \E_\omega \Big\| \sum_{j=1}^N r_j(\omega)   T_{z_j}x_j \Big\|_{\X}^2 \leq K^2  \E_\omega \Big\| \sum_{j=1}^N r_j(\omega)   x_j \Big\|_{\X}^2 
 \end{equation}
 for all integers $N$ and choices of $x_j \in \X$ and of indices $z_j$. Then the family of  trilinear forms on $ \X\times \X'\times \mathbb C$
 $$
\Lambda_z(x,x',\lambda):= \l T_z x,x'\r\lambda
 $$
satisfies $\mathcal R_{\X_1,\X_2,\X_3}(\{\Lambda_z\}) \leq 2K$. This follows immediately from the definitions and from the fact that $\mathcal R$-bounds are preserved under adjoint, so that $\{T'_{z}:z \in Z\}\subset \mathcal L(\X') $ has $\mathcal R$ bound $K$ as well. 
\subsubsection{UMD function lattices} \label{sslattices}We consider a $\sigma$-finite measure space $(\mathcal N,\nu)$ and three Banach lattices  $\X_1,\X_2,\X_3$ of measurable functions, satisfying the property
\begin{equation}
\label{latticeid}
\Big|\mathsf{id}(x_1,x_2,x_3):= \int_{\mathcal N} x_1(t)x_2(t) x_3(t)\, \d \nu(t) \Big| \leq \prod_{j=1}^3 \|x_j\|_{\X_j}.  \end{equation}
See for instance \cite{RDF86} for extensive definitions and properties of (quasi)-Banach lattice function spaces.
In particular, in this setting, if $\X_k$  has finite cotype, 
$$
 \Big\| \sum_{j=1}^N r_j x^j_k \Big\|_{\mathrm{Rad}(X_k)} \sim   \Big\| t \mapsto \Big( \sum_{j=1}^N  |x^j_k(t)|^2)\Big)^{\frac12} \Big\|_{\X_k}
$$
with comparability constant depending only on the cotype character of $\X_k$.
Further, for a function lattice $\X$, we denote by $ \X^\infty$ the completion of the normed space of finite sequences $\cic{x}=(x^1,\ldots,x^n)$, $x^j\in \X$, with respect to the norm
$$
\| \cic x\|_{\X^\infty}:= \Big\| t\mapsto \sup_{j=1,\ldots,n} |x^j(t)|\Big\|_{\X}.
$$
With these definitions it is immediately seen that, for 
  a family $\{\Lambda_{z}:z \in Z\}$ of trilinear continuous forms on  $\X_1,\X_2,\X_3$, 
we have
\begin{equation} \label{Rclarify}
\mathcal R_{\X_1,\X_2,\X_3}(\{\Lambda_z: z \in Z\}) = C \sup_\sigma
\sup_{\cic z} \|\Lambda_{\cic z}\|_{B(\mathrm{Rad}(\X_{\sigma(1)}), \mathrm{Rad}(\X_{\sigma(2)}), \X_{\sigma(3)}^\infty)}   \end{equation} 
where the family $\Lambda_{\cic z}$, indexed by finite tuples $\cic z \subset Z^{\mathbb N}$, is defined as 
  \begin{align*} &
  \Lambda_{\cic z} (\cic{x}_1, \cic x_2, \cic x_3)= \sum_{j=1}^n \Lambda_{z_j}(x_1^j,x_2^j,x_3^j),
 \\ & \cic{x}_{\sigma(1)} \in \mathrm{Rad}(\X_{\sigma(1)}), \, \cic{x}_{\sigma(2)} \in \mathrm{Rad}(\X_{\sigma(2)}),\, \cic{x}_{\sigma(3)} \in  \X_{\sigma(3)}^\infty.\end{align*}
In the relevant case where $\X_j=L^{q_j}(\nu)$ and $q_1,q_2,q_3$ is a Banach H\"older tuple, we may further identify
$$
\mathrm{Rad}(\X_{j}) \equiv L^{q_j}(\nu;\ell^2), \qquad \X_{j}^\infty \equiv L^{q_j}(\nu;\ell^\infty).
$$
We conclude that \eqref{Rclarify} is then equivalent to familiar bounds of the type
\begin{equation}
\label{Rclarify2}
\Lambda_{\cic z} (\cic{f}_1, \cic{f}_2, \cic{f}_3)  \lesssim \|\cic{f}_{\sigma(1)}\|_{L^{q_1}(\ell^2)}  \|\cic{f}_{\sigma(2)}\|_{L^{q_2}(\ell^2)} \|\cic{f}_{\sigma(3)}\|_{L^{q_3}(\ell^\infty)}
\end{equation}   
uniformly in the choice of $\cic z \in Z^{ N}$, $N\in \mathbb N$ and of permutations $\sigma$.
\section{The    model paraproducts and the nontangential RMF} \label{SecPP}
Throughout this section, we work with an (abstract) tuple  $\X_1,\mathcal X_2,\mathcal \X_3$  of Banach spaces tied by the identification \eqref{trcontfm}, which we keep implicit in our notation. Below, we will consider (singular integral) operators acting on $\X_j$-valued functions. These operators will always be initially defined for Schwartz tensors, that is functions of the type
$$
f= \sum_{\ell=1}^N  x_\ell f_\ell, \qquad x_\ell \in \X_j, \,f_j \in \mathcal S (\R^d).
$$
We denote such a class by $\mathcal S(\R^d;\X_j)$ or simply  $\mathcal S(\X_j)$ when the dimension is clear from context or unimportant. Note that $\mathcal S(\X_j)$ is dense in the Bochner spaces $L^p(\X_j)$,  $0<p<\infty$. 

\subsection{The model paraproducts} 
We are ready to set up the definition of an operator-valued paraproduct form acting on tuples of $\X_j$-valued functions. 
Let $ \cic{\Phi}$ be the collection of    (measurable) functions  $\phi$ on $\R^d$ satisfying   \begin{equation}
\label{decay} (1+|x|)^{d+1} (\partial^\alpha\phi)(x) \leq C \qquad 
\end{equation}
for all multi-indices $|\alpha| \leq N$, with $N$ large enough (in fact, $N=2$ will suffice). We adopt the notation  $$
\phi_t= \mathrm{Dil}^1_t \phi,\qquad \phi_t(x)= \frac{1}{t^d}\phi\left( \frac{x}{t}\right), \qquad x \in \R^d, t\in (0,\infty).$$
Throughout the article we use the notation
\[
B_{r}(x) =\{y\in \R^d:|y-x|<r\}, \qquad x\in \R^d,\; r>0.
\]
For a (Banach-valued) $f\in L^{1}_{\mathrm{loc}}(\R^d; \X)$ and fixed parameters $\alpha\in B_1(0)$ and $0<\beta\leq 1$, we define   strongly measurable  functions on the upper half space $\R^{d+1}_+=\R^d \times (0,\infty) $
\begin{equation}
\label{emb}
F_{\phi,\alpha,\beta}(f)(u,t)=\int_{\R^d} f(z) \phi_{\beta t} (u+\alpha t-z) \, \d z, \qquad (u,t) \in \R^{d+1}_+.
\end{equation}
Given three functions $\phi_1,\phi_2,\phi_3 \in \cic{\Phi}$ and a bounded, weakly measurable  family     $$(u,t) \in \R^{d+1}_+\mapsto \mathsf{w}(u,t)[\cdot,\cdot,\cdot] \in B(\X_1, \X_2,\X_3) $$
we consider the following model paraproduct forms, for $\pi\in \{1,2,3\}$,
\begin{align*} & \quad 
\mathsf{PP}^{\pi}_{\mathsf w}(f_1,f_2,f_3)\\ & = \int_{\R^{d+1}_+} \mathsf{w} (u,t)\left[ F_{\phi^1,\alpha_1,\beta_1}(f_1)(u,t) , F_{\phi^2,\alpha_2,\beta_2}(f_2)(u,t) ,   F_{\phi^3,\alpha_3,\beta_3}(f_3)(u,t)  \right] \, \frac{\d u \d t}{t},
\end{align*}
initially defined for tuples $f_j \in \mathcal{S} (  \X_j)$. The superscript $\pi$ stands for the  \emph{mean zero type} of the paraproduct, in the sense that  we assume that $$\widehat{\phi^{k}}(0)=0 \qquad \textrm{ for }k\in \{1,2,3\}\backslash \{\pi\}.$$  

The    $L^p$ boundedness results that follow  lie at the core of both the Coifman-Meyer type theorems of Section \ref{SecCMT} and the $T(1)$-theorems of Section \ref{T1secstat}. The RMF property appearing in the statement is extensively described in the next subsection.
 \begin{proposition}\label{thmmodel}   Let $\sigma$ be a permutation of $\{1,2,3\}$. Assume $\mathcal X_{\sigma(1)},\X_{\sigma(2)}$ are  $\mathrm{UMD}$ spaces  and that $\X_{\sigma(3)}$ has the RMF property. Further, assume that
\begin{equation}
\label{ppc}
\mathcal R_{\X_{\sigma(3)}| \X_{\sigma(1)}, \X_{\sigma(2)}}\big( \{\mathsf{w}(u,t): (u,t)\in \R^{d+1}_+\}\big)= K_\sigma <\infty.
\end{equation}

Then, for all      H\"older tuples $(p_1,p_2,p_3)$  such that
$
1< p_1,p_2,p_3 \leq \infty, 
$
and all tuples $f_j\in \mathcal S(\X_j),$ with  $j=1,2,3$, there holds
$$
 |{\mathsf{PP}^{\sigma(3)}_{\mathsf{w}}}(f_1,f_2,f_3)| \lesssim \|f_1\|_{L^{p_1}(\R^d;\X_1)} \|f_2\|_{L^{p_2}(\R^d;\X_2)}  \|f_3\|_{L^{p_3}(\R^d;\X_3)}.$$
 When either $p_{\sigma (1)}=\infty$ or $p_{\sigma (2)}=\infty$, the same result holds with the corresponding norm replaced by $\|f_{\sigma (j)}\|_{\mathrm{BMO}(\R^d,\X_{\sigma (j)})}$.

Furthermore, for all $1\leq p_1,p_2<\infty$, the weak-type estimate
$$
 |{\mathsf{PP}^{\sigma(3)}_{\mathsf{w}}}(f_1,f_2,f_3)| \lesssim \|f_1\|_{L^{p_1}(\R^d,\X_1)} \|f_2\|_{L^{p_2}(\R^d,\X_2)} |F_3|^{1-(\frac{1}{p_1}+\frac{1}{p_2})}
$$
holds for all functions $f_1\in \mathcal{S} ( \X_1) ,f_2 \in \mathcal{S} ( \X_2)$ , sets  $F_3\subset \R^d$ of finite measure, and $  \mathcal{S} ( \X_3)$-valued functions $f_3$ with $\|f_3\|_{\X_3} \leq \cic{1}_{F_3'}$, $F_3'$ being a suitable major subset of $F_3$ depending on $f_1, f_2, F_3$ and the H\"older tuple only. 
 
The implied constants depends on $ (p_1,p_2)$,  on $K_\sigma$ and $\{\beta_j:j=1,2,3\}$, as well as the $\mathrm{UMD}$ and $\mathrm{RMF}$ character of the spaces involved, and can be explicitly computed.
\end{proposition}
Proposition \ref{thmmodel} will be proved within the framework of Banach-valued outer $L^p$-spaces developed in Sections \ref{SecRM} and \ref{SecOM}. The argument is given in Section \ref{Secpfthmmodel}.

\subsection{The nontangential RMF} \label{ssRMF} In this subsection, we introduce and develop the notion of the nontangential Radema\-cher maximal function and the corresponding RMF property. Proofs of Lemmata \ref{lemmastep}, \ref{lemmaRMF1} and \ref{lemmaRMF2} below are postponed to Section \ref{ssprfrmf}. 

Denote  the \emph{cone} and \emph{truncated cone} at $s>0$, of aperture $\alpha>0$, over $y \in \R^d$ by
\begin{equation} \label{notcones}
\begin{split}
& 
\Gamma^\alpha(y)=\{(u,t) \in \R^{d+1}_+:  |u-y|<\alpha t\}.
\\
&  
\Gamma^\alpha_s(y)=\{(u,t) \in \R^{d+1}_+: 0<t<s, \, |u-y|<\alpha t\}. 
\end{split}
\end{equation}
We omit the superscript $\alpha$ when equal to 1.    
For $\phi \in \cic{\Phi}$ as in \eqref{decay} and $f\in L_{\mathrm{loc}}^1(\R^d; \X_3)$,  we define the nontangential Rademacher  maximal function  of $f$, and its \emph{grand} version   by
$$
\M_{\phi,\alpha} f(x) :=  \mathcal R_{\X_1, \X_2}\left(\big\{ f* \phi_t (y) : (y,t) \in \Gamma^\alpha(x)  \big\} \right), \quad \M_{\cic{\Phi},\alpha} f(x):=\sup_{\phi \in \cic{\Phi} } \M_{\phi,\alpha} f(x).
$$
For a fixed $p \in (1,\infty) $, we   say that $\X_3 $   has the nontangential RMF$_p$ property, with respect to the identification \eqref{trcontfm}, if$$
  \|\M_{\cic \Phi,1}\|_{L^p(\R; \mathcal X_3) \to L^p(\R) } <\infty.
$$
The following lemma essentially allows us to work with a single approximation of unity $\Psi \in \cic{\Phi}$.
\begin{lemma} \label{lemmastep}
Let $\Psi\in \cic{\Phi}$ be a   Schwartz function supported in $B_\rho(0)$, with $\int_{\R^d} \Psi=1$.  Then, for all  $f\in \mathcal S( \X_3)$    and $1\leq p\leq \infty$, we have 
 \begin{equation}
 \|\M_{\cic{\Phi},\alpha} f \|_{  L^p(\R^d)} \lesssim \|\M_{\Psi,3\rho\alpha} f \|_{  L^p(\R^d) }
 \end{equation} 
 with implied constant only depending on $C$ in \eqref{decay}, $\rho$ and  on the dimension $d$.
\end{lemma}

\begin{remark} The $L^p$ boundedness of the nontangential Rademacher maximal function does not depend on either the aperture or the ambient space dimension. In other words, 
suppose $\mathcal X_3$   has the nontangential RMF$_p$ property for some $p \in (1,\infty)$. Then Lemma \ref{lemmastep} and a simple tensor product argument imply that
$
 \|\M_{\cic{\Phi},\alpha} \|_{L^p(\R^d; \X_3) \to L^p(\R^d)}
 $ is finite as well. A dilation argument then yields that $\M_{\Psi,\alpha}$ also maps $L^p(\R^d; \X_3)$ into  $L^p(\R^d)$ with a bound independent on $\alpha$. Lemma \ref{lemmastep} then implies that the grand version $\M_{\cic{\Phi},\alpha}$ is $L^p$-bounded as well uniformly on $\alpha$. For this reason, we omit the subscript $\alpha$ and write $\M_{\cic{\Phi}}$ when $\alpha$ is unimportant or clear from the context.
\end{remark}

\begin{remark} \label{remtype2}If   $\mathcal X_1, \X_2$ have cotype 2,  as pointed out in  Remark \ref{remtype},  $\mathcal R$-bounds and  uniform bounds are equivalent. Therefore, in this case  $\M_\psi f\lesssim M (\|f\|_{\mathcal X_3})$ pointwise, where $M$ is   the standard Hardy-Littlewood maximal function. \end{remark}
The next lemma, among other implications, shows that nontangential $\mathrm{RMF}_p$ is actually $p$-independent. By virtue of this observation, in the sequel we say that $\X_3$ has   nontangential $\mathrm{RMF}$ to mean that nontangential $\mathrm{RMF}_p$  holds for some (thus for all) $p\in (1,\infty)$.
\begin{lemma} \label{lemmaRMF1} Assume $\mathcal X_3$ has the nontangential $\mathrm{RMF}_p$ property for some $p\in (1,\infty)$.
Then 
\begin{itemize}
\item[$\cdot$] $\mathcal X_3$ has the nontangential $\mathrm{RMF}_q$ property for all $q\in (1,\infty)$;
\item[$\cdot$] $\M_{\cic{\Phi}}:  H^{1}(\R^d;\X_3) \to L^{1}(\R^d ) $ for all $d\geq1$, $\alpha>0$;
\item[$\cdot$] $\M_{\cic{\Phi}}: \mathrm{BMO}(\R^d;\X_3) \to \mathrm{BMO}(\R^d) $ for all $d\geq1$, $ \alpha>0$;
\item[$\cdot$] $\M_{\cic{\Phi}}:  L^{1}(\R^d;\X_3) \to L^{1,\infty}(\R^d ) $ for all $d\geq1$, $\alpha>0$.
\end{itemize}
\end{lemma} 
In the final lemma of this section, we remark that the nontangential RMF property is enjoyed in several settings of type \eqref{trcontfm}, in particular by all (commutative and noncommutative) reflexive $L^p$ spaces.
\begin{lemma}\label{lemmaRMF2}
The Banach space $\mathcal X_3$ satisfies     the nontangential RMF property  when
\begin{itemize}
\item[(1)] $\X_3=\mathbb C$, $\X_1,\X_2$ arbitrary Banach spaces; here $\mathcal X_1$ is identified with a subspace of $(\mathcal X_2)'$;
\item[(2)]     $\X_1$, $\X_2$ have cotype 2;
\item[(3)] $\mathcal X_1, \mathcal X_2,\mathcal X_3$ are UMD Banach function lattices on the same $\sigma$-finite measure space $(\mathcal N, \nu)$ with the identification \eqref{latticeid};  \item[(4)] for $p\in (1,\infty)$,  $\X_1= \mathbb C$, $\mathcal X_2=L^{p'}(\mathcal A, \tau)$ $\X_3=  L^{p}(\mathcal A, \tau)$, with obvious \eqref{trcontfm}.  The above notation stands for the noncommutative $L^p$ spaces on a von Neumann algebra $\A$ equipped with a normal, semifinite, faithful trace $\tau$.
\end{itemize}
\end{lemma}
\begin{remark} \label{dyRMF} Below, let $\mathcal F_k$, $k \in \mathbb Z$ be the canonical dyadic filtration on $\R^d$. In previous literature on the subject, the dyadic version of $\M_\psi$
$$
\M^\triangle f(x) = \mathcal R\left(\big\{ \E(  f | \F_k) (x) : k\in \mathbb Z \big\} \right)
$$
for $f \in L^1_{\mathrm{loc}}(\R^d; {\X_3})$
has been considered instead. A Banach space ${\X_3}$ such that $\M^\triangle$  maps $L^p(\R^d; {\X_3})$ into  $L^p(\R^d)$ boundedly for some $p\in (1,\infty)$ is said to have the (dyadic) RMF property. The corresponding version of Lemma \ref{lemmaRMF1}, except for the weak-$L^1$ bound, for $\M^\triangle$ has been proved in \cite{HytMcPor08}, where the RMF property has been introduced and employed for the first time. The weak-$L^1$ bound, as well as a more systematic study which shows that the RMF property does not depend on the filtration appearing in the definition, can be found in the articles   \cite{Kemp11,Kemp13}. 

In the same article \cite{HytMcPor08}, the authors prove that the dyadic RMF property holds for 
the same classes of Banach spaces appearing in Lemma \ref{lemmaRMF2}. At present time, we are unaware of further classes of Banach spaces for which the dyadic RMF property, as well as the nontangential RMF property holds true. It seems natural to conjecture that dyadic RMF and nontangential RMF are indeed equivalent. This  is mentioned in Kemppainen's thesis \cite{KempTh} as an open question. While we are unable to derive either implication of this equivalence, the above results entail that dyadic RMF and nontangential RMF   are now known to hold in exactly the same generality. 
\end{remark}

\section{The Coifman-Meyer theorem in UMD spaces} \label{SecCMT}

In this section, we  present our main operator-valued multiplier theorem, Theorem \ref{CMT} below, and we  derive from it Corollaries \ref{introthm1} to \ref{BHT}.

Throughout, we let $\X_1,\X_2,\X_3$ be a triple of Banach spaces tied by the identification \eqref{trcontfm}. 
Let $\cic{1}=(1,1,1)\in  \R^3$ and 
$m$
be a sufficiently smooth away from the origin $ B(\X_1,\X_2,\X_3)$-valued function,  defined on the hyperplane ${\cic{1}}^\perp$. For a triple of Schwartz functions $f_j\in \mathcal S(\R;\X_j)$ $j=1,2,3$, we define the trilinear form
\begin{equation}
\label{tr}
\Lambda_m(f_1,f_2,f_3) =\int\displaylimits_{\xi \in \cic{1}^\perp} m(\xi_1,\xi_2,\xi_3) \big[ \widehat f_1(\xi_1),\widehat f_2(\xi_2), \widehat f_3(\xi_3)\big]\,\d\xi.
\end{equation}
To the  trilinear form $\Lambda_m$, we associate three dual bilinear operators $T^\sigma_m$ given by
\begin{equation}
\label{bi}
\int_{\R} \l T^\sigma_m(f_{\sigma(1)}, f_{\sigma(2)}) (x), f_{\sigma(3)}(x) \r \,\d x = 
\Lambda_m(f_1,f_2,f_3)
\end{equation}
each mapping   $ \mathcal{S}(\X_{\sigma(1)}) \times \mathcal{S}(\X_{\sigma(2)})  $  into $\mathcal{S}(\X_{\sigma(3)}')$.
We say that  $m$ is a $\mathcal R$-bounded Fourier multiplier form if
\begin{equation}
\label{derivatives}
\mathcal R_{ \X_{1 },\X_{2},\X_{3}}
\left(\big\{|\xi|^{|\alpha|} \partial_\xi^\alpha m(\xi) : \xi \in \cic{1}^\perp\big\} \right) \leq C
\end{equation}
for all multi-indices $|\alpha|\leq N$, with $N$ sufficiently large. Here, the  derivatives are taken on the hyperplane $\xi \in \cic{1}^\perp$. The condition \eqref{derivatives} can be thought of as a generalization of $\mathcal R$-boundedness of bilinear Fourier multiplier forms (i.e.\ linear Fourier multiplier operators) to the multilinear case. We are ready to state our main result.
\begin{theorem} \label{CMT} Assume, in the above setting, that each Banach space  $\X_j$ has the UMD property, and furthermore, for all cyclic permutations $\sigma$, $\X_{\sigma(3)}$ has the RMF property  corresponding to the identification \eqref{trcontfm}.
Assume that $m$ is a  $\mathcal R$-bounded Fourier multiplier form as in \eqref{derivatives}.
Then each of the bilinear operators $T^\sigma_m$ extends to a bounded bilinear operator
\begin{align*} &
T^\sigma_m: L^{p_1}(\R; \X_{\sigma(1)}) \times L^{p_2}(\R; \X_{\sigma(2)}) \to L^{r}(\R; \X_{\sigma(3)}'),   \\ & \textstyle \frac{1}{p_1} + \frac{1}{p_2} = \frac 1r, \; 1<p_1,p_2 \leq \infty, \, r<\infty; \\
&
T^\sigma_m: L^{1}(\R; \X_{\sigma(1)}) \times L^{1}(\R; \X_{\sigma(2)}) \to L^{\frac12,\infty}(\R; \X_{\sigma(3)}'),
\end{align*}
with operator norms depending only on the $\mathcal R$-bound \eqref{derivatives} of $m$, on the UMD and RMF character of the spaces $\X_j,j=1,2,3,$ and on the exponents $p_1,p_2$.
\end{theorem}
We recall that the RMF property has been defined in Subsection \ref{ssRMF}. Theorem \ref{CMT} is obtained from the model paraproduct bounds of Proposition \ref{thmmodel} by means of a standard Littlewood-Paley type decomposition argument. We sketch the details of this decomposition procedure in Section \ref{Secdec}.  \begin{remark} The derivation of Corollary \ref{introthm1} from Theorem \ref{CMT} is immediate. Indeed,  let $\X$ be a UMD space and  set $\X_1=\X,\X_2=\X',\X_3=\mathbb C$, with obvious identification \eqref{trcontfm}. Assume that both $\X$ and $\X'$ have the RMF property. If the $\mathcal{L}(\X)$-valued  multiplier $\mathsf{m}$ is as in  Corollary \ref{introthm1},  the multiplier form-valued function
$$
 (x,x',\lambda ) \in \X\times \X'\times \mathbb C \mapsto  m(\xi) [x,x',\lambda]= \lambda \l\mathsf m(\xi) x,x'\r
 $$
 satisfies assumption \eqref{derivatives} with constant equal to twice the $\mathcal R$-bound \eqref{Rboundsop}. Therefore, Corollary \ref{introthm1} follows from an application of Theorem \ref{CMT} to $\Lambda_m$. Note that  the UMD and RMF, from Lemma \ref{lemmaRMF2} (4),  assumptions are satisfied when $\X=L^p(\mathcal A, \tau)$,  the noncommutative $L^p$ spaces on a von Neumann algebra $\A$ equipped with a normal, semifinite, faithful trace $\tau$ and $1<p<\infty$. 
\end{remark}
\begin{remark}
Let us consider a triple of Banach lattices as in Paragraph \ref{sslattices}. In this setting, assuming that $\X_j$ is UMD implies automatically that each $\X_j$ has the RMF property (see Lemma \ref{lemmaRMF2}). Therefore, we are able to apply Theorem \ref{CMT} replacing the $\mathcal R$-bound \eqref{derivatives} with the lattice-type condition \eqref{Rclarify}. Corollary \ref{introthm2}, which deals with the case $\X_j=L^{q_j}(\mathcal \nu)$, is then obtained by reducing \eqref{Rclarify} further to \eqref{Rclarify2}.
\end{remark}
 
  \begin{proof}[Proof of Corollary \ref{BHT}] We prove the case where $\mathrm{dim} \,\Gamma_1=1$, the other case being identical. We recall the following result  \cite[Theorem 1.7]{Silva12}. Let $\mathcal M$ be a family of complex functions defined on $\cic{1}^\perp\subset \R^3$, smooth away from the origin and satisfying the bound
\begin{equation}
\label{bhtbdd1}
\sup_{\mathsf m \in \mathcal M}\sup_{ \eta\in \cic{1}^\perp }  \dist(\eta,\Gamma_1)^{|\beta|} \big|\partial^\beta_\eta \mathsf m( \eta)\big| \leq C_{\mathcal M}, \qquad  |\beta| \leq N
\end{equation}
Then, there holds, for each H\"older triplet with $1<q_j<\infty$ as in the statement of the corollary, 
$$
\sum_{j=1}^n\Lambda_{\mathsf{m}_j} (g_1^j, g_2^j, g_3^j) \lesssim C_{\mathcal M} \|\{g_{\sigma(1)}^j\}\|_{L^{q_1}(\R;\ell^2)}\|\{g_{\sigma(2)}^j\}\|_{L^{q_2}(\R;\ell^2)} \|\{g_{\sigma(3)}^j\}\|_{L^{q_3}(\R;\ell^\infty)}
$$
uniformly over $n$, over choices $\mathsf{m}_j \in \mathcal M$ and over permutations $\sigma$. Now, given $m$   as in \eqref{mpest}, we define for   $ \xi \in  \cic{1}^\perp$ and (initially) $g_j\in \mathcal S (\R)$, $j=1,2,3$
$$\mathsf{m}(\xi) [g_1,g_2,g_3] = \Lambda_{m(\xi,\cdot)} (g_1, g_2,g_3)=  \int\displaylimits_{\eta_1+\eta_2+\eta_3=0} m(\xi, \eta) \widehat g_1( \eta_1)\widehat g_2( \eta_2)\widehat g_3(\eta_3)\,   \d \eta.
$$
We finally learn from \eqref{bhtbdd1} and assumption \eqref{mpest} that the family 
\begin{equation}
\label{bht3}
\mathcal M:= \left\{\mathsf{m}(\xi), |\xi|^\alpha \partial_\xi^\alpha \mathsf{m}(\xi): \xi \in \cic 1^\perp, |\alpha |\leq N \right\} 
\end{equation}
satisfies \eqref{condLqj} with $C_{\mathcal M} \leq K$.

 We move to the core of the proof. By mixed norm interpolation (see \cite{BenPan61} and the recent account \cite[Section 3]{Silva12}), it suffices to prove the weak-type bound with $p_1=p_2=1$. Fix $f_j \in L^1(\R; L^{q_j}(\R)) $, $j=1,2$ of unit norm. 
 We generally denote
$$
F_j: \R \to L^{q_j}(\R), \qquad F_j(x) = \left\{ y \mapsto f_j(x,y) \right\}.
$$ Fix  $E_3\subset \R$. By virtue of $\mathcal M$ from \eqref{bht3} satisfying \eqref{condLqj} with $C_\mathcal M=K$, we can apply (the dual form of) Corollary \ref{introthm2} and   learn that there exists a suitable major subset $E'_3$ of $E_3$ such that for all  functions $f_3$ with $\|F_3(x)\|_{q_3} \leq \cic{1}_{E_3'}(x)$, we have,  
$$
|\Lambda_{m}(f_1,f_2,f_3)| = |\Lambda_{\mathsf m}(F_1,F_2,F_3)| \lesssim |E_3|^{-1}   K
$$
which  is exactly the generalized restricted weak-type $(1,1,-1)$ bound for  the form $\Lambda_m$ dual to $T$. This completes the proof.
\end{proof}

\section{Operator-valued $T(1)$ theorems in one and multiple parameters} \label{T1secstat}

In this section we collect the main results pertaining operator-valued $T(1)$-type theorems we described in the introduction. We stress that  these results are obtained within the Banach-valued outer measure theory framework of Section \ref{SecOM}. Throughout this section, $\X_1,\X_2$ will be UMD Banach  spaces over $\mathbb C$.

\subsection{$T(1)=0$ theorems in one and multiple parameters}
 Let $\phi:\R \to \R$ be a fixed smooth  function  with $\widehat\phi(0)=0$ and define for $z=(x,s)\in \R^2_+$\[  \phi_{z}=\phi_{x,s}=s^{-1}\phi(s^{-1}(\cdot-x)).
\]
We have the following theorem. The proof is given in Section \ref{SecT1pf}.
\begin{theorem}[$T(1)=0$ theorem]\label{onepara}
Let $\Lambda: \mathcal{S}(\R;\X_1)\times \mathcal{S}(\R;\X_2)\rightarrow \mathbb C$ be a bilinear form. For any fixed $z=(x,s),w=(y,t)\in\R^2_+$, define $Q_{z,w}:\X_1\times\X_2\rightarrow \mathbb C$ as
\[
Q_{z,w}(\xi,\eta):=A(z,w)\Lambda(\phi_{z}\xi,\phi_{w}\eta),\quad\forall\,\xi\in\X_1,\eta\in\X_2,
\]
where
\begin{equation}
\label{Adef}
A(z,w)=A(x,s,y,t):=\frac{\max(s,t,|x-y|)^2}{\min(s,t)}.
\end{equation}
Assume  the randomized boundedness condition 
\begin{equation}\label{newwbp}
\mathcal{R}_{\X_1,\X_2}(\{Q_{z,w}: \,z,w\in\R^2_+\})=K<\infty.
\end{equation}
Then, for $1<p<\infty$, $f_j\in \mathcal S(\R;\X_j)$, $j=1,2$,
\[
|\Lambda(f_1,f_2)|\leq C_pK\|f_1\|_{L^p(\R;\X_1)}\|f_2\|_{L^{p'}(\R;\X_2)},
\]
for some constant $C_p$ depending only on $\phi$, $p$ and on the UMD character of $\X_1,\X_2$.\end{theorem}
\begin{remark} The conclusion of the above theorem implies that $\Lambda$ satisfying \eqref{newwbp} extends, by density, to an element of $B(L^p(\R;\X_1),L^{p'}(\R;\X_2) )$, for all $1<p<\infty$. A similar remark can be made for Theorems \ref{bipara}, \ref{FullT1thm}, and \ref{biparafull} below.
\end{remark}
\begin{remark}[Comparison to other $T(1)=0$ theorems] Let  $T$ be  the (distributional) dual operator to the form $\Lambda$ appearing in
Theorem \ref{onepara}. Our result can be used to recover the  operator-valued $T(1)$ theorem (on the real line) in its usual formulation, in the paraproduct-free case. Results of this type  date back to \cite{HytWeis08}, while a more  recent formulation is given in \cite{HytHann14}. The typical assumptions of these theorems include a $\mathcal R$-bounded version of the standard kernel estimates and of the  weak boundedness property, besides $T(1)=T^*(1)=0$. It is not hard to check that our condition (\ref{newwbp}) follows from these two assumptions. Indeed, when $|x-y|$ is much larger than both $s$ and $t$, (\ref{newwbp}) can be deduced from the (randomized) standard kernel estimates. When $s$ or $t$ dominates, (\ref{newwbp}) follows from the kernel estimates combined with the condition $T(1)=T^*(1)=0$. The other cases can be obtained using the (randomized) weak boundedness property. 
 \end{remark}

{{ 
\begin{remark}
If we define $Q_{z,w}(\xi,\eta)$ using
\[
A^\delta (z,w)=A^{\delta}(x,s,y,t):=\frac{\max(s,t,|x-y|)^{1+\delta}}{\min(s,t)^{\delta}}
\]
instead, then it can be easily checked that Theorem \ref{onepara} holds true for all $1/2\leq\delta\leq 1$. Moreover, this range can be further extended to $0<\delta\leq 1$ if $\X_1$, $\X_2$   both   have type $2$. Similar remarks apply  for Theorems \ref{bipara}, \ref{FullT1thm}, and \ref{biparafull} below.
\end{remark}
}}

We now iterate Theorem \ref{onepara}   to obtain the following mixed norm, bi-parameter operator-valued $T(1)$ theorem.   The assumption of Pisier's property $(\alpha)$ on the Banach spaces, see Lemma \ref{lemmaalpha}, is necessary, as it can be inferred from \cite{HytPor08}, where mixed norm boundedness results for multi-parameter convolution-type singular integrals and pseudodifferential operators have been established. The theorem obviously extends to arbitrarily many parameters.

Let $z_j=(x_j,s_j)\in \R^2_+,$ $j=1,2,$  $\phi_{z_1}$ be the same as above, and $\psi_{z_2}$ be another bump function in the second variable with the same properties.

\begin{theorem}[Bi-parameter $T(1)=0$]\label{bipara}
Let $\X_1,\X_2$ be UMD spaces with {{property ($\alpha$)}}. 
Let $$\Lambda: \mathcal{S}(\R)\otimes\mathcal{S}(\R;\X_1)\times \mathcal{S}(\R)\otimes \mathcal{S}(\R;\X_2)\rightarrow \mathbb C$$ be a bilinear form. For any fixed $z_1 ,z_2 ,w_1,w_2 \in\R^2_+$, define $Q_{z_1,z_2,w_1,w_2}:\X_1\times\X_2\rightarrow \mathbb C$ as
\[
Q_{z_1,z_2,w_1,w_2}(\xi,\eta):=A(z_1,w_1)A(z_2,w_2)\Lambda(\phi_{z_1}\otimes\psi_{z_2}\xi,\phi_{w_1}\otimes \psi_{w_2}\eta),
\]
for $\xi\in\X_1,\eta\in\X_2.$
Assume that
\begin{equation}\label{biparanewwbp}
\mathcal{R}_{\X_1,\X_2}(\{Q_{z_1,z_2,w_1,w_2},\,z_1,z_2,w_1,w_2\in\R^2_+\})=K<\infty.
\end{equation}
Then, for $1<p,q<\infty$,  $f_j\in  \mathcal{S}(\R)\otimes \mathcal{S}(\R;\X_j)$, $j=1,2$, 
\[
|\Lambda(f_1,f_2)|\leq C_{p,q}K\|f_1\|_{L^p(\R;L^q(\R;\X_1))}\|f_2\|_{L^{p'}(\R;L^{q'}(\R;\X_2))},
\]
for some constant $C_{p,q}$ depending only on $\phi$, $\psi$, $p$, $q$ and on the UMD character of $\X_1,\X_2$.
\end{theorem}  

\begin{remark}[Comparison to other scalar-valued bi-parameter $T(1)$ theorems]
When $\Lambda$ is   associated with a scalar distribution kernel, Theorem \ref{bipara} recovers  the known $T(1)$ theorems in the bi-parameter setting with the additional assumption that $T$ is paraproduct free, i.e. $T$ together with all its partial adjoints map tensor products of $1$ with bump functions to $0$. For example, assume $T$ is as considered in \cite{PottVilla11} or \cite{Mar12}, satisfying mixed type standard kernel estimates, weak boundedness properties and is paraproduct free. Then, a similar argument as in the one-parameter setting shows that condition (\ref{biparanewwbp}) holds for bilinear form $\langle Tf,g\rangle$. {{The same holds true in the multi-parameter setting, i.e. our result in its multi-parameter version recovers the $T(1)$ theorem in \cite{Ou} for paraproduct free operators.}}
\end{remark}

\subsection{$T(1)$ theorems in one and multiple parameters}

The $T(1)=0$ type Theorems \ref{onepara} and \ref{bipara} can be combined with model  paraproduct estimates to obtain full $T(1)$-type results. Again, neither the conditions on the cancellative part nor the definition of $T(1)$ involve any assumption on the distributional kernel.
 
In the one-parameter setting, we can employ our model paraproduct Theorem \ref{thmmodel} to prove a fully operator-valued $T(1)$ theorem under a suitable assumption on the spaces involved. In the statement below, we let $\rho$ be a nonnegative Schwartz function  with $\rho(0)=1$, and $\phi$ be as in Theorem \ref{onepara}.   The proof is a simple combination of Theorem \ref{onepara} and the bound of Theorem \ref{thmmodel}.
\begin{theorem}\label{FullT1thm}
Let $\X_1,\X_2,\X_3$ be a triple of UMD Banach spaces coupled via an identification $\mathsf{id}$ as in \eqref{trcontfm}. Assume that both $\X_1$ and $\X_2$ have the nontangential RMF property with respect to the identifications \eqref{trcontfm}.
 Let $ \Lambda$ be a bilinear   form on  $\mathcal S(\R;\X_1) \times \mathcal S(\R;\X_2)  $. 
Assume that  the limits below exist and are finite for every $\xi_j \in \X_j$, $j=1,2,$ and $f\in \mathcal S(\R)$ and the equalities
\begin{equation}
\label{T1def}
\begin{split} &
\lim_{s\to \infty} \Lambda( \rho_{0,s} \xi_1, f \xi_2) = \int_{\R} \mathsf{id}(\xi_1,\xi_2, b(x)) f(x) \, \d x, \\ & \lim_{s\to \infty} \Lambda( f  \xi_1, \rho_{0,s}\xi_2) = \int_{\R} \mathsf{id}(\xi_1,\xi_2, b^\star(x)) f(x) \, \d x,
\end{split}
\end{equation}
hold for some $b,b^\star\in \mathrm{BMO}(\R; \X_3)$. Furthermore, assume that the bilinear continuous form
$$
\Lambda_0(f_1,f_2) = \Lambda(f_1,f_2) - \mathsf{PP}_{\mathsf{id}}^{1}(f_1,f_2,b) -\mathsf{PP}_{\mathsf{id}}^2(f_1,f_2,b^\star).
$$
satisfies the assumption \eqref{newwbp} of Theorem \ref{onepara} for a suitable choice of auxiliary functions $\phi^j$, $j=1,2,3$. Then, for all $1<p<\infty$, $\Lambda$ extends to a bounded bilinear form on ${L^p(\R;\X_1)}  \times {L^{p'}(\R;\X_2)}$, 
$$
|\Lambda(f_1,f_2)| \lesssim \|f_2\|_{L^p(\R;\X_1)}\|f_2\|_{L^{p'}(\R;\X_2)}.
$$
The implied constant depends on the size of \eqref{newwbp}, on $\|b\|_{\mathrm{BMO}(\R; \X_3)}$, $\|b^\star\|_{\mathrm{BMO}(\R; \X_3)}$, and on the UMD and RMF characters of $\X_1,\X_2, \X_3$.
\end{theorem}
\begin{remark}
Due to the need for the RMF assumption to control the continuous paraproducts, which seems unavoidable, Theorem \ref{FullT1thm} does not recover the most general operator-valued $T(1) \in $ $\mathrm{BMO}$-type theorems, see for instance the recent formulation of \cite{HytHann14}. Results in the spirit of \cite{HytHann14}, relying on the kernel dyadic representation theorems, proceed through bounds for paraproducts of discrete type and are not affected by this difficulty.

However, relying on point (4) of  Lemma \ref{lemmaRMF2}, we are able to tackle the case of $\X_1$ being a noncommutative $L^q$-space with $1<q<\infty$, $\X_2=\X_1'$, and $\X_3=\mathbb C$ (that is, $b$ is scalar valued). Also in view of Lemma \ref{lemmaRMF2}, Theorem \ref{FullT1thm} applies with no restriction in the UMD lattice setting of  Paragraph \ref{sslattices}.
\end{remark}

In the bi-parameter setting, Theorem \ref{bipara} together with bi-parameter paraproduct estimates implies a fully bi-parameter scalar-valued $T(1)$ theorem. In the following, $\rho^1,\rho^2$ are fixed nonnegative Schwartz functions with $\rho^i(0)=1,i=1,2$. And $\Lambda$ is a continuous bilinear form on $\mathcal{S}(\R)\otimes\mathcal{S}(\R)\times\mathcal{S}(\R)\otimes\mathcal{S}(\R)$. Assume that the limits below exist finite for every bump function $\phi^1,\phi^2\in\mathcal{S}(\R)$ with mean zero and the equalities
\begin{equation}
\label{biparaT1def1}
\begin{split}
&\lim_{s_1,s_2\rightarrow \infty}\Lambda(\rho^1_{0,s_1}\otimes\rho^2_{0,s_2},\phi^1\otimes\phi^2)=\int_{\R\times\R} b_1(x,y)\phi^1(x)\phi^2(y)\,\d x\d y,\\
&\lim_{s_1,s_2\rightarrow \infty}\Lambda(\phi^1\otimes\phi^2,\rho^1_{0,s_1}\otimes\rho^2_{0,s_2})=\int_{\R\times\R}b_2(x,y)\phi^1(x)\phi^2(y)\,\d x\d y,\\
&\lim_{s_1,s_2\rightarrow\infty}\Lambda(\rho^1_{0,s_1}\otimes \phi^2,\phi^1\otimes\rho^2_{0,s_2})=\int_{\R\times\R} b_3(x,y)\phi^1(x)\phi^2(y)\,\d x\d y,\\
&\lim_{s_1,s_2\rightarrow \infty}\Lambda(\phi^1\otimes\rho^2_{0,s_2},\rho^1_{0,s_1}\otimes\phi^2)=\int_{\R\times\R} b_4(x,y)\phi^1(x)\phi^2(y)\,\d x\d y
\end{split}
\end{equation}
hold for some $b_1,b_2,b_3,b_4\in  \mathrm{BMO}(\R\times\R)$, where this means Chang-Fefferman's product BMO \cite{CFBMO}. Moreover, assume that the limits below exist and are  finite for every bump function $\phi\in\mathcal{S}(\R)$ with mean zero and $f_1,f_2\in\mathcal{S}(\R)$ and the equalities
\begin{equation}
\label{biparaT1def2}
\begin{split}
&\lim_{s_1\rightarrow\infty}\Lambda(\rho^1_{0,s_1}\otimes f_1,\phi\otimes f_2)=\int_{\R\times\R} b_5(x)\phi(x)f_1(y)f_2(y)\,\d x\d y,\\
&\lim_{s_2\rightarrow\infty}\Lambda(f_1\otimes\rho^2_{0,s_2},f_1\otimes\phi)=\int_{\R\times\R} b_6(y)\phi(y)f_1(x)f_2(x)\,\d x\d y,\\
&\lim_{s_1\rightarrow\infty}\Lambda(\phi\otimes f_1,\rho^1_{0,s_1}\otimes f_2)=\int_{\R\times\R} b_7(x)\phi(x)f_1(y)f_2(y)\,\d x\d y,\\
&\lim_{s_2\rightarrow\infty}\Lambda(f_1\otimes\phi,f_2\otimes\rho^2_{0,s_2})=\int_{\R\times\R} b_8(y)\phi(y)f_1(x)f_2(x)\,\d x\d y
\end{split}
\end{equation}
hold for some $b_5,b_6,b_7, b_8\in \mathrm{BMO}(\R)$. Furthermore, assume that the bilinear continuous form
\[
\Lambda_0(f,g):=\Lambda(f,g)-\mathsf{BPP}(f,g,b_1,b_2,b_3,b_4,b_5,b_6,b_7,b_8)
\]
satisfies the assumption (\ref{biparanewwbp}) in Theorem \ref{bipara}, {{which in the current scalar setting reads as
\[
\sup_{z_1,z_2,w_1,w_2\in\mathbb{R}^2_+}|A(z_1,w_1)A(z_2,w_2)\Lambda_0(\phi_{z_1}\otimes\psi_{z_2},\phi_{w_1}\otimes\psi_{w_2})|<\infty.
\]}}The form $\mathsf{BPP}$ is the sum of all the bi-parameter paraproducts discussed in \cite{PottVilla11}. Recall that in the bi-parameter setting, there are three different types of paraproducts: standard bi-parameter paraproduct, mixed bi-parameter paraproduct, and the partial paraproduct, which is essentially a paraproduct in one-parameter. A typical standard bi-parameter paraproduct appearing in our $\mathsf{BPP}$ is
\[
B^1(b_1,f,g)=\sum_{R}\langle b_1, \psi_R\rangle\langle f,\psi_R^2\rangle\langle g,\psi_R\rangle,
\]
where $\{\psi_R=\psi_{R_1}\otimes\psi_{R_2}\}_R$ is a wavelet basis on $L^2(\R^2)$, and $\psi_R^2=\psi_{R_1}^2\otimes\psi_{R_2}^2$, with $\psi_I^2$ being a bump function adapted to $I$ such that $\hat{\psi_I^2}$ has compact support in a set whose measure is comparable with $|I|^{-1}$. Similarly, a typical mixed bi-parameter pararoduct and a partial paraproduct appearing in our $\mathsf{BPP}$ are as the following:
\[
\begin{split}
&B^2(b_3,f,g)=\sum_R \langle b_3, \psi_R\rangle\langle f, \psi_{R_1}^2\otimes \psi_{R_2}\rangle\langle g,\psi_{R_1}\otimes \psi_{R_2}^2\rangle,\\
&B^3(b_5,f,g)=\sum_R\langle b_5, \psi_{R_1}\rangle\langle f, \psi_{R_1}^2\otimes \psi_{R_2}\rangle\langle g, \psi_{R_1}\otimes \psi_{R_2}\rangle.
\end{split}
\]
We refer the readers to \cite{PottVilla11} for more detailed discussions of these paraproducts and their boundedness estimates. The following theorem can be proved by combining the paraproduct estimates and Theorem \ref{bipara}.

\begin{theorem}\label{biparafull}
Let $\Lambda: \mathcal{S}(\R)\otimes\mathcal{S}(\R)\times\mathcal{S}(\R)\otimes\mathcal{S}(\R)\rightarrow \mathbb C$ be a bilinear form satisfying the above properties. Then, for all $1<p<\infty$, $\Lambda$ extends to a bounded bilinear form on $L^p(\R^2)\times L^{p'}(\R^2)$.
\end{theorem}

\section{Randomized norms} \label{SecRM} We now  lay the foundation of the Banach-valued outer measure spaces of  Section \ref{SecOM}.
Here,  $Z$ will stand for a metric space equipped with a $\sigma$-finite Borel measure $\upsilon$. 
In this section  ${\X_1,\X_2,\X_3}$ are three complex Banach spaces with the identification \eqref{trcontfm}. Furthermore, we assume that $\X_1,\X_2$ are UMD spaces. For a strongly $\upsilon$-measurable function $F: Z\to {\X_3}$, we define the (quasi)-norm
\begin{equation} \label{Rmu}
\|F  \|_{R({\upsilon},{\X_3})} := \mathcal R_{\X_1,\X_2}\big( \{ F(z): z\in Z  \}\big).
\end{equation}
We recall  the following well-known properties of the $R({\upsilon},{\X_3})$ (quasi)-norm defined above. For references, we send to \cite[Proposition 2.1]{Gaans06} for the first three, and to \cite[Lemma 5.8]{KalWeis14} for the last one.
\begin{lemma}\label{Rbounds}Let  $F,G: Z\to{\X_3}$ bounded and strongly measurable, $\lambda>0$, $m\in L^\infty(Z)$, $h \in L^1(Z)$. Then
\begin{align} & \|\lambda F\|_{{R}({\upsilon}, {\X_3})} = |\lambda| \| F\|_{{R}({\upsilon}, {\X_3})}, \\
& \| F+ G \|_{{R}({\upsilon}, {\X_3})} \leq  2\|F\|_{{R}({\upsilon}, {\X_3})}+2\|G\|_{{R}({\upsilon}, {\X_3})}, 
\\&
\big\| z\mapsto m(z) F(z) \big\|_{{R}({\upsilon}, {\X_3})} \leq 2\|m\|_\infty\|F \|_{{R}({\upsilon}, {\X_3})},\label{therest}
\\
& 
\| F*h\|_{{R}({\upsilon}, {\X_3})} \leq 2\|h\|_1 \| F \|_{{R}({\upsilon}, {\X_3})}.
\end{align} 
\end{lemma}


\subsection{Square function norms}

Let now  $\X$ be a  UMD Banach space. For  simple and $L^2$ normal\-ized functions in $L^2(Z,{\upsilon}) \otimes \mathcal X$, that is functions of the form
\begin{equation}f=\sum_{n=1}^N  \frac{\cic{1}_{A_n}}{\sqrt{{\upsilon}(A_n)}} x_n, 
\label{simple}
\end{equation}
where $x_n \in \mathcal X $ and $A_n\subset Z$  pairwise disjoint sets of  finite nonzero measure, we define the norms
\begin{equation}
\label{randnorms}
\left\|\sum_{n=1}^N  \frac{\cic{1}_{A_n}}{\sqrt{{\upsilon}(A_n)}} x_n\right\|_{S({\upsilon},\mathcal X)}:=\left( \E \left\|\sum_{n=1}^N r_n x_n \right\|_\X^2 \right)^{\frac12}.
\end{equation}
We call ${S}({\upsilon}, \mathcal X)$ the Banach space obtained by completion of the simple functions in $L^2(Z,{\upsilon}) \otimes \mathcal X$ with respect to the above norm.    
We summarize some   properties of the ${S}({\upsilon}, \mathcal X)$-norms in the  remarks and lemmata that follow.
\begin{remark}  The UMD property yields in particular that $\X$ is reflexive and has finite cotype. These two properties imply that $f: Z\to \mathcal X$ strongly ${\upsilon}$-measurable belongs to $S({\upsilon},\X)$ if and only if for any   orthonormal basis  $\{h_n\}$ in $L^2(Z,{\upsilon})$, 
$$
\left( \mathbb E \left\|  \sum_{n} r_n \l f,h_n \r_{L^2(Z,{\upsilon})}\right\|^2_\X \right)^{\frac12} <\infty
$$
and the above quantity is comparable to the norm defined in \eqref{randnorms} with a constant independent of the choice of the basis (and of the function $f$). See \cite[Section 4]{KalWeis14}.  
We work with simple functions to avoid technicalities in the definition of a (smooth) orthonormal system. 

 Moreover, $\{r_n\}$ can be replaced by a standard independent Gaussian sequence $\{\gamma_n\}$, in the sense that
\begin{equation}\label{comparison}
\|f \|_{{S}({\upsilon}, \mathcal X)} \sim \|f\|_{\gamma(L^2(Z,{\upsilon}),\X) }:= \sup \left( \mathbb E \left\|  \sum_{n} \gamma_n \l f,h_n \r_{L^2(Z,{\upsilon})}\right\|^2_\X \right)^{\frac12},
\end{equation} the
supremum above being taken over all orthonormal systems $\{h_n\}$ in $L^2(Z,{\upsilon})$. This follows, for instance, from \cite [Corollary 3.6, Theorem 3.7]{Nee08}. The $\gamma(L^2(Z,{\upsilon}),\X)$ norm is more widely used in preexisting literature on Banach space-valued square functions; see, for instance, \cite{KalWeis14,HytNeePor08,HytWeis10} and references therein.
\end{remark}
\begin{remark} When $\X$ is a Hilbert space, it is clear by orthonormality of the  $r_n$ that  $S( {\upsilon},X)$ is isometrically isomorphic to the Bochner space $L^2(Z,{\upsilon};\mathcal X)$.
\end{remark}
\begin{lemma}\cite[Corollary 4.8, 4.9a]
{KalWeis14} \label{kw48} Let $m\in L^\infty(Z)$, $F: Z\to\mathcal X $ strongly measurable. Then
$$
\big\| z\mapsto m(z) F(z) \big\|_{{S}({\upsilon}, \mathcal X)} \leq 2 \|m\|_\infty\|F \|_{{S}({\upsilon}, \mathcal X)}.
$$
\end{lemma}

\subsection{H\"older inequality for randomized norms}  The following two lemmata play the role of the (respectively) Cauchy-Schwarz and of the H\"older inequality for the randomized norms defined above. We omit the simple proof of the first lemma, which can be obtained along the same lines of Lemma \ref{innerholder} below.
\begin{lemma} 
\label{innercs}  Let $F_k:Z \to \X_k$, $k=1,2$, be strongly $\upsilon$-measurable, bounded functions. Let $ {\Lambda}: Z \to \B( \X_1,\X_2) $ be a weakly measurable family of bilinear continuous forms. Then, referring to \eqref{rbound-ab},   $$
\int_Z \left|{\Lambda}_z\big(F_1(z),F_2(z)\big)\right| \d {\upsilon}(z) \leq 2 \mathcal R_{\\X_1,\X_2} (\{\Lambda_z:z\in Z\}) \|F_1\|_{S({\upsilon},\X_1)} \|F_2\|_{S({\upsilon},\X_2)}  .
 $$
\end{lemma} 
%
 \begin{lemma} 
\label{innerholder} In the above setting, let $F_k:Z \to \X_k$, $k=1,2$, $F_3: Z\to {\X_3}$ be strongly $\upsilon$-measurable, bounded functions. Let $ {\Lambda}: Z \to \B( \X_1,\X_2,{\X_3}) $ be a weakly measurable family of trilinear continuous forms. Then, referring to \eqref{Rlambda},   $$
\int_Z \left|{\Lambda}\big(F_1,F_2, F_3\big)\right| \d {\upsilon} \leq 2 \mathcal R_{{\X_3}|\X_1,\X_2} (\{\Lambda_z:z\in Z\}) \|F_1\|_{S({\upsilon},\X_1)} \|F_2\|_{S({\upsilon},\X_2)}  \|F_3\|_{R({\upsilon},{\X_3})}.
 $$
\end{lemma}
\begin{proof}[Proof of Lemma \ref{innerholder}]Appealing to  strong measurability   we can assume $F_j$, $j=1,2,3$ are simple valued and  write  
$$
F_k(z)=\sum_{j=1}^N \frac{x^j_{k}}{\sqrt{{\upsilon}(A_j)}} \cic{1}_{A_j}(z), \; k=1,2,  \quad F_3(z)=\sum_{j=1}^N  {x_3^{j}}  \cic{1}_{A_j}(z),  $$
with $A_j \subset Z$ of finite nonzero measure. Subsequently, we can employ  Fatou's lemma and approximate ${\Lambda}(F_1,F_2,F_3)$ by simple functions   ${\Lambda}_z \equiv {\Lambda}_{z_j}$ for  $z \in A_j$.  Then, for suitable unimodulars $\alpha_j$, $j=1,\ldots,N$, \begin{equation}
\label{avertrick}\nonumber
\begin{split} &\quad\int_Z \left|{\Lambda}_z\big(F_1(z),F_2(z), F_3(z)\big)\right| \d {\upsilon}(z)= \sum_{j=1}^N {\Lambda}_{z_j} (\alpha_j x^j_1,x^j_2,x_3^{j})  \\ &   \leq \mathcal R_{{\X_3}|\X_1,\X_2} (\{\Lambda_z\})  
 \RR_{\X_1,\X_2}\left(\{x_3^{1},\ldots,x_3^{N}\} \right)  \Big(\E \Big\| \sum_{j=1}^N r_j  \alpha_j  x_1^j \Big\|_{\X_1}^2 \Big)^{\frac12}\Big(\E  \Big\| \sum_{j=1}^N r_j    x_2^j \Big\|_{\X_2}^2 \Big)^{\frac12}  
\end{split}
\end{equation}
and the conclusion follows by using Kahane's contraction principle, to deal with the $\alpha_j$, and by recalling the definitions of the randomized norms.
\end{proof}

\section{Banach-valued outer measure theory and conical sizes} \label{SecOM}
 In this paragraph, we extend the outer measure spaces and outer $L^p$-spaces introduced in \cite{DoThiele15} to Banach space-valued functions. Aside from a couple of technical points which we stress in the definitions below, the construction proceeds along the same lines.
\subsection{Outer measure, sizes and outer $L^p$}
Throughout this section, without further explicit mention,  $\mathcal X$ will be a complex Banach space.
Let $Z$ be a metric space and $\mathbf{E}$ a subcollection of  Borel subsets of $Z$. Let ${\sigma}$ be a premeasure, that is, a function from $\mathbf{E}$ to $[0,\infty)$. We say that the premeasure ${\sigma}$ generates the outer measure   \cite[Definition 2.1]{DoThiele15}   $\mu$ defined by
\[
Z\supset E \mapsto \mu(E):=\inf_{\mathbf{E}'}\sum_{E'\in\mathbf{E}'} \sigma(E') 
\]
where the infimum is taken over all subcollections $\mathbf{E}'$ of  $\mathbf{E}$ which cover the set $E$. To simplify the statements that follow, we work with subcollections $\mathbf{E}$ possessing the following property: for every point $z \in Z$ one may find an open ball $B$ contaning $z$ and a set $E \in \cic{E}$ with $B \subset E$. This geometric constraint will be apparently satisfied in our concrete case.

Denote by $\mathcal{B}(Z;\X)$  the set of $\X$-valued, strongly Borel-measurable   functions on $Z$. A \emph{size} is a map
\[
\mathsf{s}:\, \mathcal{B}(Z;\X)\rightarrow [0,\infty]^{\mathbf{E}}
\]
satisfying for any $f,g$ in $\mathcal{B}(Z;\X)$ and $E\in\mathbf{E}$  
\begin{align}
& \label{qmp}   \mathsf{s}(z\mapsto m(z)f(z))(E)\leq  2\mathsf{s}(f)(E) \qquad \forall m: Z\to \mathbb C, \,\|m\|_\infty \leq 1,\\\
&\label{sizelin} \mathsf{s}(\lambda f)(E)=|\lambda|\mathsf{s}(f)(E),\quad\forall \lambda\in\mathbb{C},\\
& \label{sizesubadd}  \mathsf{s}(f+g)(E)\leq C\mathsf{s}(f)(E)+C\mathsf{s}(g)(E),\quad\text{for some fixed $C\geq 1$}.
\end{align}  

We say that the triple $(Z, \sigma,\mathsf{s})$  is a $\X$-valued \emph{outer measure space}, keeping the generating collection  $\cic{E}$ and the Banach space  implicit in the notation. If the  underlying premeasure ${\sigma}$ is obvious we simply write $(Z,\mathsf{s})$ for the corresponding $\X$-valued  {outer measure spaces} instead.
\begin{remark}Observe how the quasi-monotonicity property \eqref{qmp} replaces   the corresponding monotonicity property of the scalar-valued case. The factor $2$ could be discarded from \eqref{qmp} if $\X$ were a real Banach space. We prefer to work in the complex case; dealing with this extra factor does not introduce essential changes in what follows.   Property \eqref{sizesubadd} above is 
referred to as \emph{quasi-subadditivity}.
\end{remark}

We now move onto the definition of $\X$-valued \emph{outer $L^p$ spaces}. First of all, we define the \emph{outer (essential) supremum} of $\mathcal{B}(Z;\X)$ over the Borel-measurable   $F\subset  Z$ by\[
\underset{F}{\text{outsup}\,}\mathsf{s}(f):=\sup_{E\in\mathbf{E}}\mathsf{s}(f\cic{1}_{ F})(E).
\]
Note that the monotonicity property (1) of $\mathsf{s}$ implies that $\text{outsup}_{F}\mathsf{s}(f)\leq 2\text{outsup}_{G}\mathsf{s}(f)$ whenever $F\subset G$.
We then define, for each $\lambda>0$, the \emph{super level measure} $$
\mu(\mathsf{s}(f)>\lambda)= \inf \left\{ \mu(F): F\subset Z \textrm{ Borel}, \, \underset{Z\backslash F}{\text{outsup}\,}\mathsf{s}(f) \leq \lambda\right\}. $$  
Finally, for     $f\in \mathcal{B}(Z;\X)$  we define 
\begin{align*}
&\|f\|_{L^p(Z,\mathsf{s})}:=\left(\int_0^\infty p\lambda^{p-1}\mu(\mathsf{s}(f)>\lambda)\,\d\lambda\right)^{\frac1p}, \qquad 0<p<\infty,\\ &
\|f\|_{L^{p,\infty}(Z,\mathsf{s})}:= \sup_{\lambda>0} \lambda\big(\mu(\mathsf{s}(f)>\lambda)\big)^{\frac1p}, \qquad 0<p<\infty,
\\ &
\|f\|_{L^{\infty,\infty}(Z,\mathsf{s})}=\|f\|_{L^\infty(Z,\mathsf{s})}:=\underset{Z}{\text{outsup}\,}\mathsf{s}(f)=\sup_{E\in\mathbf{E}}\mathsf{s}(f)(E).
\end{align*}
Arguing along the lines of \cite[Proposition 3.1]{DoThiele15}, in view of properties \eqref{qmp}-\eqref{sizesubadd} of the size, we obtain the following basic properties of outer $L^p$ spaces.
\begin{proposition} Let $(Z,{\sigma}, \mathsf{s})$ be an $\X$-valued outer measure space, $f,g \in \mathcal B(Z; \mathcal X)$. Then, for $0<p\leq \infty$, we have
\begin{align}
& \label{Lpqmp}   \|mf\|_{L^{p}(Z,{\sigma}, \mathsf{s})}\leq 2 \|f\|_{L^{p}(Z,\sigma, \mathsf{s})} \qquad \forall m: Z\to \mathbb C, \,\|m\|_\infty \leq 1,\\\
&\label{Lplin} \|\lambda f\|_{L^{p}(Z,{\sigma}, \mathsf{s})}=|\lambda| \|  f\|_{L^{p}(Z,{\sigma}, \mathsf{s})},  \qquad\forall \lambda\in\mathbb{C},
\\
&\label{Lpscal} \|  f\|_{L^{p}(Z,\lambda{\sigma}, \mathsf{s})}=\lambda^{1/p} \|  f\|_{L^{p}(Z,{\sigma}, \mathsf{s})},  \qquad\forall \lambda>0,\\
& \label{Lpsubadd} \|  f+g\|_{L^{p}(Z, {\sigma}, \mathsf{s})}\leq C \|  f\|_{L^{p}(Z, {\sigma}, \mathsf{s})}+ C\|  g\|_{L^{p}(Z, {\sigma}, \mathsf{s})},
\end{align}
and identical statements hold with $L^{p,\infty} $ replacing ${L^{p}}$ in \eqref{Lpqmp}-\eqref{Lpsubadd}.
\end{proposition}

\begin{remark}Let $(Z,\sigma,\mathsf{s})$ be an outer measure space and $W\subset Z$ be a Borel set. It will be convenient to consider the  restriction to $W$ of the corresponding  outer $\X$-valued $L^p$ spaces defined  by  
$$
\|f\|_{L^{p}(W,{\sigma}, \mathsf{s})}:= \|f\cic{1}_{W}\|_{L^{p}(Z,{\sigma}, \mathsf{s})}
$$
It follows by \eqref{Lpqmp} that $\|f\|_{L^{p}(W,{\sigma}, \mathsf{s})}\leq 2\|f \|_{L^{p}(Z,{\sigma}, \mathsf{s})} $, and similarly for outer $L^{p,\infty}$-spaces. \end{remark}

\begin{proposition}[Marcinkiewicz interpolation] \label{propint} Let    $(Z,\mathsf{s})$ be an $\X$-valued  outer measure space.  Let $\nu$ be a  $\sigma$-finite positive Borel measure,  and let $T$ be a quasi-sublinear operator mapping  functions in the Bochner spaces $L^{p_1}(\nu; \X)$ and $L^{p_2}(\nu;\X)$ into  strongly measurable, $\X$-valued functions on $Z$, for some $1\leq p_1<p_2\leq\infty$. Assume that for all $f\in L^{p_1}(\nu;\X))+L^{p_2}(\nu;\X))$,
\begin{align*}
&\|T(f)\|_{L^{p_1,\infty}(Z,\mathsf{s})}\leq A_1\|f\|_{L^{p_1}(\nu;\X)},\\
&\|T(f)\|_{L^{p_2,\infty}(Z,\mathsf{s})}\leq A_2\|f\|_{L^{p_2}(\nu;\X))}.
\end{align*}
Then there holds
\[
\|T(f)\|_{L^p(Z,\mathsf{s})}\leq A_1^{\theta_1}A_2^{\theta_2}C_{p_1,p_2,p}\|f\|_{L^p(\nu; \X)},
\]
for all $p_1<p<p_2$, with $\theta_1,\theta_2$ such that $\theta_1+\theta_2=1$ and $1/p=\theta_1/p_1+\theta_2/p_2$.
\end{proposition}
We sketch the proof of the following important proposition, which combines (a vector-valued version of) \cite[Proposition 3.4]{DoThiele15} with the Radon-Nykodym property \cite[Proposition 3.6]{DoThiele15}.
\begin{proposition}[H\"older's inequality]\label{holderDT} Let $Z$ be a metric space and ${\nu} $ be a positive Borel measure on $Z$. Let $\sigma$ be an outer measure on $Z$ generated   by the collection $\cic{E}$ via the premeasure ${\sigma}$, and let $(Z,{\sigma},\mathsf{s})$ be a scalar-valued outer measure space  with the property that $$
\int_{E} |f(z)| \,\d{\nu}(z) \leq K_1 {\sigma}(E) \mathsf{s}(f) (E) \qquad \forall E \in \cic{E}. 
$$
For $j=1,\ldots,n$, let  $\X_j $ be a complex Banach space and $(Z,{\sigma},\mathsf{s}_j)$
 be an $\X_j$-valued outer measure space. Suppose that for a weakly (Borel)-measurable family of sub-$n$-linear forms
$$
(x_1,\ldots,x_n)
\in \X_1\times\cdots\times \X_n \mapsto \Lambda_z(x_1,\ldots,x_n) \in \mathbb C,\qquad z\in Z
$$ there holds
\begin{equation} \label{sizereq}
\mathsf{s}\big(z\mapsto |\Lambda_z(F_1(z),\ldots,F_n(z))| \big)(E)\leq K_2\prod_{j=1}^n \mathsf{s}_j(F_j)(E)\qquad \forall E \in \cic{E}
\end{equation}
and all bounded  $F_j\in\mathcal{B}(Z; \X_j)$, $j=1,\ldots,n$.
 Then, for all Borel   $W\subset Z$
\begin{equation}
\label{hineq}
\int_{W} \left|\Lambda_z\big(F_1(z),\ldots,F_n(z)\big)\right| \, \d {\nu}(z) \leq n2^nK_1K_2 \prod_{j=1}^n \|F_j\|_{L^{p_j}(W,\sigma,\mathsf{s}_j)}
\end{equation}
for all  tuples $(p_1,\ldots,p_n)$ with $1\leq p_1,\ldots, p_n\leq \infty$ and $\frac{1}{p_1}+\cdots+\frac{1}{p_n}=1$.
\end{proposition}
\begin{proof} By replacing each $F_j$ with $F_j\cic{1}_{W}$, it suffices to argue in the case $W=Z$. Let then   $G= |\Lambda(F_1,\ldots,F_n)|$. By \cite[Proposition 3.6]{DoThiele15}, we have
$$
\int_Z G(z) \, \d \upsilon(z) \leq K_1 \|G\|_{L^{1}(Z,\sigma,\mathsf{s})}.
$$ Hence, it will suffice to bound the left hand side of the above display by the right hand side of 
\eqref{hineq}. Arguing exactly as in   \cite[Proposition 3.4]{DoThiele15}, this is an easy consequence of the inequality
\begin{equation}
\label{hineq2}
\mu(\mathsf{s}(G) >2^nK_2\lambda) \leq \sum_{j=1}^n \mu\left(\mathsf{s}_j(F_j) >\lambda^{\frac{1}{p_j}}\right),
\end{equation}
which is proved as follows: in fact, by scaling, it suffices to prove the case $\lambda=1$. Given any $\eps>0$, we may find sets $H_j$, $j=1,\ldots,$ with $$\mu(H_j)\leq \mu(\mathsf{s}_j(F_j) >1) +\textstyle\eps,\qquad \displaystyle\sup_{E \in \cic{E}} \mathsf{s}_j(F_j\cic{1}_{Z\backslash H_j})(E) \leq 1. $$ Then, setting  $H=\cup_j H_j$, we have, by assumption and quasi-monotonicity \eqref{qmp} of the size
$$
  \mathsf{s}(G\cic{1}_{Z\backslash H})(E)\leq K_2 \prod_{j=1}^n \mathsf{s}_j(F_j{1}_{Z\backslash H})(E)
  \leq 2^n \prod_{j=1}^n \mathsf{s}_j(F_j{1}_{Z\backslash H_j})(E)\leq 2^nK_2 \qquad  \forall E \in \cic{E},
$$
whence
$$
\mu(\mathsf{s}(G) >2^n\lambda) \leq \mu(H) \leq  \mu(H_j) \leq n\eps + \sum_{j=1}^n \mu(\mathsf{s}_j(F_j) >1).
$$
Since $\eps>0$ is arbitrary, \eqref{hineq2} follows and the proof is complete.
\end{proof}

\subsection{Outer measure from tents; lacunary and overlapping sizes}
With reference to the previous subsections, we now adopt the standing notations
\begin{equation} \label{themeas}
(Z,{\upsilon})=\left(  \R^d_u \times (0,\infty)_t , \textstyle\frac{\d u \d t}{t^{d+1}}\right).  \end{equation}
Our outer measure  $\mu$  will be defined on the upper half space $Z$ via the premeasure ${\sigma}$ assigning to each tent over a cube $\mathsf{T}(Q)$ the volume of $Q$. See below for precise definitions.
The measure ${\upsilon}$ will not be employed in the definition of $\mu$ but rather in the construction of  suitable sizes involving the truncated cones $\Gamma^s(y)$ defined in \eqref{notcones}.
\subsubsection*{Tents}  For $(x,s) \in Z$ define
$$
\mathsf{T}(x,s)=\big\{(y,t) \in Z, 0 <t< s, |y-x| \leq s-t\big\}
$$
and set $\sigma(\mathsf{T}(x,s))=s^d$. In general, if $I \subset \R^d$ is a cube centered at $c_I$ and of sidelength $\ell (I)$, we write $\mathsf{T}(I)$ for $\mathsf{T}(c_I, \ell( I))$. The collection of tents forms a generating collection $\mathbf{E}$ of our outer measure space with premeasure $\sigma$. 
 \subsubsection*{Lacunary sizes} Let $\mathsf{T}=\mathsf{T}(x,s)$ be a tent and $\X$ be a  Banach space. For $F: Z\to \mathcal \X$ strongly-${\upsilon}$ measurable    and $0<q< \infty$,   we define the \emph{lacunary sizes}
\begin{align*}  &
\mathsf{s}_{q}^{\X}(F) (\mathsf{T})= \left(\frac{1}{s^d} \int_{B_{s}(x)} \|F\cic{1}_{\Gamma_{s}(y)}\|^q_{  S ({\upsilon},\X)} \d y\right)^{\frac1q},\qquad  0<q<\infty.\end{align*}
If no more than one Banach space is involved we discard the superscript and write $\mathsf{s}_{q}$ for $ \mathsf{s}_{q}^{\X}.$  
Notice that $\mathsf{s}_q$ defined above satisfy the properties, with $C=1$, of the definition in the previous paragraph: this simply follows from $1$-homogeneity and   triangle inequalities for the norms of $  S ({\upsilon},\X)$ and  $L^q(B_{s}(x))$.  Furthermore, the monotonicity property \eqref{qmp} is a consequence  of Lemma \ref{kw48}.  

{{} While} it is often convenient to work with a fixed $0<q<\infty$, the analogue of John-Nirenberg's inequality in this context, summarized in the lemma that follows, yields that outer $L^p$-spaces corresponding to different values of $q$ are equivalent. We omit the argument, which has countless analogues in literature: see for instance \cite[Theorem 2.7]{MuscSchlII}.
\begin{proposition} \label{JNp}Let $\X$ be a Banach space and $0<q_1<q_2<\infty$. Then, for all $F:Z\to \mathcal X$  strongly measurable and bounded
 $$
\sup_{\mathsf{T} }  \mathsf{s}_{q_2}(F) (\mathsf{T}) \leq C_{q_1,q_2}  \sup_{\mathsf{T} } \mathsf{s}_{q_1}(F) (\mathsf{T}), \qquad C_{q_1,q_2}= 12^{\frac{1}{q_1}}\cdot 8^{\frac{1}{q_2}}\cdot 4^{\frac {q_2}{q_1}} 
  $$
  the suprema being taken over all tents $\mathsf T(x,s)$ with $(x,s)\in Z$. As a consequence,
{{}  
$$
\|F\|_{L^{p}(Z,{\mathsf s_{q_2})}} \leq C_{q_1,q_2}  \|F\|_{L^{p}(Z,{\mathsf s_{q_1}})}.
$$}
 \end{proposition}

 \subsubsection*{Overlapping sizes}
Fixing a triple of Banach spaces $\X_1,\X_2,{\X_3}$ in the same setting as Subsection \ref{ss11},  define the \emph{overlapping sizes}\footnote{We borrow the \emph{lacunary} (mean zero component of the paraproduct) and \emph{overlapping} (mean one component of the paraproduct)  terminology from the usual dictionary of time-frequency analysis, originating in \cite{LTC}.} 
\begin{align*} &
\mathsf{r}_q(F) (\mathsf{T})= \left(\frac{1}{s^d} \int_{B_{s}(x)} \|F\cic{1}_{\Gamma_{s}(y)}\|^q_{  R ({\upsilon},{\X_3})} \d y\right)^{\frac1q}, \qquad 0<q<\infty, \\ &   \mathsf r_\infty(F) (\mathsf{T})= \sup_{y\in  B_{s}(x) }\|F\cic{1}_{\Gamma_{s}(y)}\|_{R({\upsilon},{\X_3})}.
\end{align*}
Again, the fact that $\mathsf{r}_q$ defined above satisfy properties \eqref{sizelin} and \eqref{sizesubadd}, with $C=2$, of the definition of size follows from $1$-homogeneity and quasi-triangle inequalities for the $\mathcal R$-bounds, see Lemma \ref{Rbounds},     and  for $L^q(B_{s}(x))$.  The monotonicity property \eqref{qmp} is  a consequence  of the third assertion of  Lemma \ref{Rbounds}.
\subsubsection*{H\"older's inequality for   tent outer measure spaces}
Combining Lemma \ref{innercs} and \ref{innerholder}, respectively, with the classical H\"older inequality and Proposition \ref{holderDT} we reach the  Cauchy-Schwarz and  H\"older's inequalities below, involving Banach valued outer $L^p$ norms. We only prove Proposition \ref{outerholder}, the proof of Proposition \ref{outerCS} being extremely similar.
\begin{proposition} \label{outerCS}     Let $F_k:Z \to \X_k$, $k=1,2$ be strongly $\upsilon$-measurable, bounded functions and    ${\Lambda}=\Lambda_{z}: Z \to \B( \X_1,\X_2)  $ be a weakly measurable and  bounded family of bilinear continuous forms.  Then for any $W\subset Z$ Borel measurable, and $
1\leq q \leq \infty$,
$$
\int_{W} |\Lambda_{(u,t)} \left( F_1(u,t) ,F_2(u,t)  \right)|\, \frac{\d u \d t}{t}  \leq 64 \mathcal R(\{\Lambda_z:z \in W\})  \|F_1\|_{L^{p}(W,  \mathsf s_{q}^{\X_1})} \|F_2\|_{L^{p'}(W,\mathsf s_{q'}^{\X_2})}.   
$$
\end{proposition}

\begin{proposition} \label{outerholder}     Let $F_k:Z \to \X_k$, $k=1,2$, $F_3: Z\to {\X_3}$ be strongly $\upsilon$-measurable, bounded functions and ${\Lambda}=\Lambda_{z}: Z \to \B( \X_1,\X_2,{\X_3}) $ be a weakly measurable, bounded family of trilinear continuous forms. Then for any $W\subset Z$ Borel measurable, 
\begin{align*}
&\quad \int_{W} |\Lambda_{(u,t)} \left( F_1(u,t) ,F_2(u,t) , F_3(u,t)  \right)|\, \frac{\d u \d t}{t} \\ &\leq 2^9 \mathcal R(\{\Lambda_z:z \in W\})  \|F_1\|_{L^{p_1}(W,  \mathsf s_{q_1}^{\X_1})} \|F_2\|_{L^{p_2}(W,\mathsf s_{q_2}^{\X_2})} \|F_3\|_{L^{p_3}(W,\mathsf r_{q_3}^{{\X_3}})} 
\end{align*}
for all H\"older tuples 
$$
1\leq p_j,q_j\leq \infty, \qquad \textstyle  \frac{1}{p_1}+\frac{1}{p_2}+\frac{1}{p_3}= \frac{1}{q_1}+\frac{1}{q_2}+\frac{1}{q_3}=1.
$$
\end{proposition}
\begin{proof}As usual, it suffices (by possibly replacing $F_j$ by their restrictions to $W$) to prove  the case $W=Z$. Let $\mathsf s$ be the  size  defined on scalar functions by
$$
\mathsf s(G) (\mathsf{T}(x,s)) := \frac{1}{s^d} \int_{\mathsf{T}(x,s)} |G(u,t)|\, \frac{\d u \d t}{t}.
$$ 
This is the $S_1$ appearing in \cite{DoThiele15}.
By Fubini's theorem, we have
$$
\mathsf s(G) (\mathsf{T}(x,s)) \leq \frac{1}{s^d} \int_{B(x,s)} \int_{\Gamma_s(y)} |G(u,t)|\, \frac{\d u \d t}{t^{d+1}}\, \d y,
$$
Now if $G=\Lambda(F_1,F_2,F_3)$, applying in cascade Lemma \ref{innerholder} and the classical H\"older inequality with tuple $q_j$, the above display is bounded by $$
2 \mathcal R(\{\Lambda_z:z \in Z\}) \mathsf s_{q_1}^{\X_1}(F_1) (\mathsf{T}(x,s))\mathsf s_{q_2}^{\X_2}(F_2) (\mathsf{T}(x,s))\mathsf r_{q_3}^{{\X_3}}(F_3) (\mathsf{T}(x,s)).
$$
This verifies the assumption \eqref{sizereq}. Applying Proposition \ref{holderDT} with tuple $p_j$, we complete the proof.
\end{proof}

\section{Lacunary Carleson embedding theorems}\label{SecLCET}
This section, as well as the next,  contain respectively \emph{lacunary} and \emph{overlapping} Carleson embedding type results, in the outer measure theory formulation, for the upper half-space extensions of Banach-valued functions defined in \eqref{emb}.

 Throughout this section, we work with a fixed UMD   space $\X$, and indicate by 
 $\tau=\tau(\X)$ the  (nontrivial) type of $\X$. We write $\mathsf{s}_q$ for ${\mathsf s}_q^\X$, and $L^{p} ({\mathsf s_q})$ for the outer $L^p$ spaces $L^{p}(\R^{d+1}_+,{\mathsf s_q})$ constructed in Section \ref{SecOM}.
\begin{proposition}[Carleson Embedding Theorems] \label{cet}
Let $\phi:\R^d \to \R $ be a function of the class  $\cic{\Phi}$ defined by \eqref{decay}. Assume that $\phi$ has mean zero. For $f \in \mathcal{S}(\R^d; \X)$, $\alpha \in B_1(0),  0<\beta\leq 1$, define 
$$
F_{\phi,\alpha,\beta}(f)(u,t)= f* \phi_{\beta t} (u+\alpha t).$$ For $1\leq q<\infty$, set \begin{equation}
\label{theeps0}
\eps_0=\eps_0(\mathcal X,q):=   \frac{1}{\min\{q,\tau\}}-\frac12 .
\end{equation}
Then, for all $\eps>\eps_0$,
\begin{align}
\label{embSq} &\|F_{\phi,\alpha,\beta} (f)\|_{L^{\infty}({\mathsf s_q})} \lesssim \beta^{-d\eps}\|f\|_{\mathrm{BMO}(\R^d; \X)},  \\
& \|F_{\phi,\alpha,\beta} (f)\|_{L^{p} ({\mathsf s_q})} \lesssim  \beta^{-d\eps} \|f\|_{L^p(\R^d; \X)},\qquad q<p<\infty,\label{LpSqemb}
\\ & 
\label{LqinfSq}
\|F_{\phi,\alpha,\beta} (f)\|_{L^{q,\infty} ({\mathsf s_q})} \lesssim  \beta^{-d\eps} \|f\|_{L^q(\R^d;\X)},
\\
&\|F_{\phi,\alpha,\beta} (f)\|_{L^{p} ({\mathsf s_q})} \lesssim  \beta^{-d(\theta+\eps)} \|f\|_{L^p(\R^d;\X)},\quad 1<p<q, \quad \theta=\frac{1/p- 1/q}{1- 1/q}, \label{LpSqbelow}\\
&\|F_{\phi,\alpha,\beta} (f)\|_{L^{1,\infty} ({\mathsf s_q})} \lesssim    \beta^{-d(1+\eps)} \|f\|_{L^1(\R^d; \X)}.\label{L1Sqemb}
\end{align}
The implied constants may depend on $\eps,q,p$, and on the UMD character of $\X$ only.
\end{proposition} In fact, a byproduct of the proof of  \eqref{L1Sqemb} (compare with   definition \eqref{weakL1future} below) is 
\begin{proposition}[Carleson embedding theorem, local version] \label{cetloc} Let  $f\in \mathcal S(\R^d; \X)$ and set  \begin{equation}
\label{maxfctbd} \cic{Q}_{f,\lambda}= \textrm{\emph{max.\ dyad.\ cubes }} Q \textrm{\emph{ s.t. }}
 3Q\subset \{x\in\R^d: \mathrm{M}(\|f\|_{\X})(x)>\lambda\}
\end{equation}
  {{} where $\mathrm M$ stands for the usual Hardy-Littlewood maximal operator on $\R^d$}. Denote {{} \begin{equation}
\label{exccoll}\mathsf T_{\beta}(\cic{Q}_{f,\lambda}):=\bigcup_{Q \in \cic{Q}_{f,\lambda}}  \mathsf T(c_Q, 3\beta^{-1}\ell(Q)).\end{equation} }
  Under the same assumptions of Proposition \ref{cet}, for all $\eps>\eps_0$ as in   \eqref{theeps0} and $1<q<\infty$ 
$$\|F_{\phi,\alpha,\beta} (f)\|_{L^{\infty}(Z\setminus \mathsf T_ {\beta}(\cic{Q}_{f,\lambda}),\mathsf{s}_q)} \lesssim \beta^{-d\eps}\lambda$$
with implicit constant depending on $\eps,q$ and on the UMD character of $\X$ only. 
\end{proposition}
The lacunary Carleson embedding theorems above  can be thought of  as local version of Proposition \ref{HyNe} below, a result proved in  \cite{HytNeePor08}, which is relied upon in the proof. This proposition  is a UMD-valued extension of the  Hilbert-valued tent space estimates by Harboure, Torrea, and Viviani \cite[Theorem 2.1]{HarTorViv91}. \begin{proposition}\label{HyNe}
Let $\X$ be a UMD space of type $\tau$, and $f\in \mathcal S(\R^d;\X)$, $1<q<\infty$. For all $\eps>\eps_0$, where $\eps_0$ is as in \eqref{theeps0}, there holds
\[
\left(\int_{\R^d}\|F_{\phi,0,1}(f) \cic{1}_{\Gamma^{1/\beta}(y)}\|^q_{ S ({\upsilon},\X)}\,\d y\right)^{1/q}\lesssim_q \beta^{-d\left(\frac12+\eps\right)}\|f\|_{L^q(\R^d;\X)}.
\]
\end{proposition}
We send to  \cite{HytNeePor08} and references therein for the proof. Here, we remark that, if $\X$ is a Hilbert space and $q\geq 2$, one may take $\eps=\eps_0$: see \cite{HarTorViv91}.

\begin{remark} We can upgrade \eqref{LqinfSq} to the strong type. Fixing $1<q<\infty$, and $\eps>\eps_0$, denote by $\eps'=(\eps+\eps_0)/2$.   Interpolating, via Proposition \ref{propint}, \eqref{LpSqemb} for $p=p_2>q$ and \eqref{LpSqbelow} with $\eps'$ in place of $\eps$ and    $p=p_2<q$ chosen such that $\theta=\eps'-\eps$, we obtain
\begin{equation}
\label{LqSqemb}
\|F_{\phi,\alpha,\beta} (f)\|_{L^{q}(\mathsf s_{q})} \lesssim  \beta^{-d\eps} \|f\|_{L^q(\R^d;\X)},
\end{equation}
\end{remark}
\begin{remark} When $\X$ has type 2, and $q\geq 2$, the constant $\eps_0$ defined in \eqref{theeps0} is equal to zero. If $\X$ is a Hilbert space, as it will be clear from the proof, all the estimates of Proposition \ref{cet} hold for  $\eps=0$ as well. This is not the case for the strong $L^q(\mathsf s_q)$ estimate \eqref{LqSqemb}. This proposition thus improves on the corresponding scalar-valued version of \cite[Lemma 4.3]{DoThiele15}.
\end{remark}
We now pass to the proof of Proposition \ref{cet}, which is divided into several steps. \subsection{Preliminary reductions}The first is the simple observation that, by outer measure interpolation, it suffices to prove  \eqref{embSq}, \eqref{LqinfSq}, and \eqref{L1Sqemb}. Notice that \eqref{LqinfSq}, \eqref{LqSqemb}, and \eqref{LpSqbelow} are stronger than the results one might obtain directly from interpolating (\ref{embSq}) and (\ref{L1Sqemb}). The second reduction is that
 we will assume in the proof that the bump function $\phi$ is compactly supported on $B_1(0)$. The general case can then be recovered via the, by now standard, mean zero with rapid decay decomposition of \cite[Lemma 3.1]{MPTT06}. 

\subsection{Reduction to tilted sizes}\label{sstilt}  Fixing $\alpha\in B_1(0)$, $0<\beta\leq 1$, one can define a tilted tent $\mathsf{T}_{\alpha,\beta}(x,s)$ containing those points $(u,t)$ such that $(u-\alpha\beta^{-1}t,\beta^{-1}t)\in \mathsf T(x,s)$. Similarly, for any $y\in B_s(x)$, the associated tilted cone $\Gamma_{s,\alpha,\beta}(y)$ is defined to contain points $(u,t)$ such that $(u-\alpha\beta^{-1}t,\beta^{-1}t)\in\Gamma_s(y)$. Note that $\Gamma_{s,\alpha,\beta}$ is not a symmetric cone. For $0<q<\infty$, we define the \emph{lacunary tilted sizes}
\[
{\mathsf s_{q,\alpha,\beta}}(F) (\mathsf T_{\alpha,\beta}(x,s))= \left(\frac{1}{s^d} \int_{B_{s}(x)} \|F\cic{1}_{\Gamma_{s,\alpha,\beta}(y)}\|^q_{\mathcal S ({\upsilon},\X)} \d y\right)^{\frac1q} \; (0<q<\infty).
\]
Notice that when $\alpha=0,\beta=1$, ${\mathsf s_{q,\alpha,\beta}}=\mathsf s_q$, the standard size. Via a straightforward change of variable, for any Borel measurable function $F$ on $\R^{d+1}_+$,
\[
\mathsf{s}_{q}(F(u+\alpha t,\beta t))(T(x,s))=\beta^{d/2}{\mathsf s_{q,\alpha,\beta}}(F(u,t))(\mathsf T_{\alpha,\beta}(x,s)).
\]
We now construct the outer measure $\mu_{\alpha,\beta}$ and the  outer $L^p$ spaces $L^p(\R^{d+1}_+; \mathsf s_{q,\alpha,\beta})\equiv L^p(  \mathsf s_{q,\alpha,\beta}) $ using the tilted tents $\mathsf  T_{\alpha,\beta}$ as our generating collection, with premeasure ${\sigma}_{\alpha,\beta}(\mathsf T_{\alpha,\beta}(x,s)):=s^d$ and  size ${\mathsf s_{q,\alpha,\beta}}$. By transport of structure
\[
\|F_{\phi,\alpha,\beta}(f)\|_{L^p( \mathsf s_q)}=\beta^{d/2}\|F_{\phi,0,1}\|_{L^p(  \mathsf s_{q,\alpha,\beta})},
\]
and the same for the weak norms. Hence, the estimates  \eqref{embSq}, \eqref{LqinfSq}, and \eqref{L1Sqemb} of Theorem \ref{cet} can be obtained  respectively from
\begin{align}
\label{embSq1} &\|F_{\phi} (f)\|_{L^{\infty}( \mathsf s_{q,\alpha,\beta})} \lesssim  \beta^{-d(\frac12+\eps)}\|f\|_{\mathrm{BMO}(\R^d; \X)},  \\
&\label{LqSqemb1} \|F_{\phi}(f)\|_{L^{q,\infty}(\mathsf s_{q,\alpha,\beta})} \lesssim \beta^{-d(\frac12+\eps)} \|f\|_{L^q(\R^d;\X)}, \\
&\|F_{\phi} (f)\|_{L^{1,\infty}(\mathsf s_{q,\alpha,\beta})} \lesssim C_{q}  \beta^{-d(\frac32+\eps)} \|f\|_{L^1(\R^d; \X)} \label{L1Sqemb1}
\end{align}
where we have agreed to write $F_{\phi}(f)$ for the non-tilted function $F_{\phi,0,1}(f)$.


\subsection{Proof of the $L^\infty$ estimate (\ref{embSq1})}According to the definition of the outer $L^\infty$ norm, we need to show that for all $x,s$,
\[
\frac{1}{s^d}\int_{B_s(x)}\|f\ast \phi_t(u)\cic{1}_{\Gamma_{s,\alpha,\beta}(y)}\|^q_{S ({\upsilon},\X)}\,\d y\lesssim \beta^{-qd\left(\frac12+\eps\right)}\|f\|_{\mathrm{BMO}}^q.
\]
Observe that $\Gamma_{s,\alpha,\beta}(y)\subset \Gamma_{\beta s}^{2/\beta}(y)$. Letting $Q=B_{4s}(x)$ and $f_Q$ the average value of $f$ on $Q$, the LHS of the above is
\[
\begin{split}
 &\leq \frac{1}{s^d}\int_{B_s(x)}\|f\ast \phi_t(u)\cic{1}_{\Gamma_{\beta s}^{2/\beta}(y)}\|^q_{S ({\upsilon},\X)}\,\d y \\ &\leq \frac{1}{s^d}\int_{B_s(x)}\|[(f-f_Q)\cic{1}_Q]\ast \phi_t(u)\cic{1}_{\Gamma^{2/\beta}(y)}\|^q_{S ({\upsilon},\X)}\,\d y,
\end{split}
\]
where we have used the fact that the union of $\Gamma_{\beta s}^{2/\beta}(y)$ over $y\in B_s(x)$ is contained in the enlarged non-tilted tent $\mathsf T(x,4s)$ and $\phi$ has compact support.
Therefore, due to Proposition \ref{HyNe} and the John-Nirenberg inequality, the above is
\[
\lesssim\beta^{-qd\left(\frac12+\eps\right)}\frac{1}{s^d}\|(f-f_Q)\cic{1}_Q\|_{L^q(\X)}^q\lesssim \beta^{-qd\left(\frac12+\eps\right)}\|f\|_{\mathrm{BMO}}^q,
\]
which completes the proof.

\subsection{Proof of (\ref{LqSqemb1}):  weak-$L^q$ estimate}
%


Assume $\|f\|_{L^q(\R^d;\X)}=1$. Given  $\lambda>0$, it suffices to find a set $E\subset \R^{d+1}_+$ such that 
\begin{equation}\label{weakLp}
{\mu}_{\alpha,\beta}(E)\lesssim \lambda^{-q}\beta^{-qd\left(\frac12+\eps\right)}
\end{equation}
and
\begin{equation}\label{weakLp1}
{\mathsf s_{q,\alpha,\beta}}(F_{\phi}(f)\cic{1}_{E^c})(\mathsf T_{\alpha,\beta}(x,s))\lesssim \lambda,\qquad \forall \mathsf T_{\alpha,\beta}(x,s).
\end{equation}
Let $\mathbf{B}$ be the collections of balls $B_i=B_{s_i}(x_i)$ such that
\[
\frac{1}{{s_i}^d}\int_{B_{s_i}(x_i)}\|F_{\phi}(f)\cic{1}_{\Gamma_{s_i,\alpha,\beta}(y)}\|^q_{S ({\upsilon},\X)}\,\d y>\lambda^q.
\] 
Let $E=\bigcup_{i: B_i\in\mathbf{B}}\bigcup_{y\in B_{s_i}(x_i)}\Gamma_{s_i,\alpha,\beta}(y)$. By construction, property (\ref{weakLp1}) is satisfied. It thus suffices to verify (\ref{weakLp}).
By a standard Besicovitch-type covering argument, the centers of the balls $B_i\in\mathbf{B}$ can be covered by finitely many countable subcollections $\mathbf{B}_1,\ldots,\mathbf{B}_n$, such that 
\[
\bigcup_{i:B_i\in\mathbf{B}}B_i=\bigcup_{i: B_i\in\mathbf{B}_1}B_i \cup\ldots\cup\bigcup_{i: B_i\in\mathbf{B}_n} B_i
\]
 and each of the $\mathbf{B}_n$ consists of pairwise disjoint balls. We  study  $\mathbf{B}_1$ below, and  analogous estimate will hold for all the other sub-collections. Observe that for all $y\in B_s(x)$, $\Gamma_{s,\alpha,\beta}(y)\subset \mathsf T_{\alpha,\beta}(x,Cs)$ for some fixed constant $C$. Hence,
\[
\begin{split}
&\quad \lambda^q{\mu}_{\alpha,\beta}\Big(\bigcup_{i: B_i\in\mathbf{B}_1} \bigcup_{y\in B_{s_i}(x_i)}\Gamma_{s_i,\alpha,\beta}(y)\Big)
\leq \lambda^q\sum_{i:B_i\in\mathbf{B}_1}{\mu}_{\alpha,\beta}\Big(\bigcup_{y\in B_{s_i}(x_i)}\Gamma_{s_i,\alpha,\beta}(y)\Big)\\
&\leq\lambda^q\sum_{i:B_i\in\mathbf{B}_1}{\mu}_{\alpha,\beta}(\mathsf T_{\alpha,\beta}(x_i,Cs_i))\lesssim \sum_{i: B_i\in\mathbf{B}_1} 
\int_{B_{s_i}(x_i)} \|F_{\phi}(f)\cic{1}_{\Gamma^{2/\beta}(y)}\|^q_{S ({\upsilon},\X)}\,\d y,
\end{split}
\]
according to the construction of $\{B_i\}$ and the fact that $\Gamma_{s_i,\alpha,\beta}(y)\subset \Gamma^{2/\beta}(y)$. Since all the balls $B_i\in\mathbf{B}_1$ are disjoint, the above is
\[
\leq \int_{\R^d} \|F_{\phi}(f)\cic{1}_{\Gamma^{2/\beta}(y)}\|^q_{S ({\upsilon},\X)}\,\d y\lesssim \beta^{-qd\left(\frac12+\eps\right)},
\]
where the last inequality follows from Proposition \ref{HyNe}.

\subsection{{{}Proof of (\ref{L1Sqemb1}): weak-$L^1$ estimate}}
We can assume that $\|f\|_{L^1(\X)}=1$. Then in order to show the weak-$L^1$ bound, for any $\lambda>0$, it suffices to find a set $E\subset \R^{d+1}_+$ with the property that ${\mu}_{\alpha,\beta}(E)\lesssim  \lambda^{-1}\beta^{-d}$ and such that
\begin{equation}\label{weakL1}
{\mathsf s_{q,\alpha,\beta}}(F_{\phi}(f)\cic{1}_{E^c})(\mathsf T_{\alpha,\beta}(x,s))\lesssim \lambda \beta^{-d\left(\frac12+\eps\right)} \qquad \forall\, \mathsf T_{\alpha,\beta}(x,s).
\end{equation}
We define such a set by  \begin{equation}\label{weakL1future}
E=\bigcup_{Q \in \cic{Q}_{f,\lambda}}\mathsf T_{\alpha,\beta}(c_Q, 3\beta^{-1}\ell(Q))\end{equation} where    $\cic{Q}_{f,\lambda}$  is defined  in \eqref{maxfctbd}. Using that the elements of $\cic{Q}_{f,\lambda}$   are pairwise disjoint thus yields
\begin{align*}
&\quad {\mu}_{\alpha,\beta}(E)\leq  \sum_{Q \in \cic{Q}_{f,\lambda}} \sigma_{\alpha,\beta}(\mathsf T_{\alpha,\beta}(c_Q, 3\beta^{-1}\ell(Q)) \lesssim \beta^{-d}\sum_{Q \in \cic{Q}_{f,\lambda}} |Q| \\ &\lesssim \beta^{-d} |\{x\in\R^d: \mathrm{M}(\|f\|_{\mathcal X})(x)>\lambda\}| \lesssim \beta^{-d} \lambda^{-1}.
\end{align*} 
We move to the proof of (\ref{weakL1}). The Calder\'on-Zygmund decomposition yields a  good function $g$ and bad functions $b_Q, {Q \in \cic{Q}_{f,\lambda}}$ such that \[
f=g+b, \qquad b=\sum_{Q \in \cic{Q}_{f,\lambda}}\ b_Q\] with the following properties:
\begin{align}
&\|g\|_{L^\infty(\X)}\lesssim \lambda;\\ \label{CZD}
& \text{supp}\,b_Q\subset Q, \qquad \int b_j=0,\qquad \|b_Q\|_{L^1(\X)}\lesssim \lambda |Q|.
\end{align}
For an arbitrary tilted tent $\mathsf T_{\alpha,\beta}(x,s)$, one has by estimate (\ref{embSq1}) that
\begin{equation} \label{Aug150}
{\mathsf s_{q,\alpha,\beta}}(F_{\phi}(g))(\mathsf T_{\alpha,\beta}(x,s))\lesssim \beta^{-d(\frac{1}{2}+\eps)}\|g\|_{L^\infty(\X)}\lesssim \lambda \beta^{-d(\frac{1}{2}+\eps)}.
\end{equation}
We now move to controlling the contribution of the bad part $b$. We will prove  that for
\[
P(u,t):=  F_{\phi}(b)\cic{1}_{E^c}(u,t)= \sum_{{Q \in \cic{Q}_{f,\lambda}} } P_Q(u,t), \qquad P_Q(u,t):= F_{\phi}(b_Q)\cic{1}_{E^c}(u,t)
\]
there holds 
\begin{equation} \label{Aug151}
\sup_{(x,s)\in Z}   \mathsf{s}_{1,\alpha,\beta}  (P)(\mathsf T_{\alpha,\beta}(x,s) ) \lesssim  \lambda\beta^{-\frac d2}.
\end{equation}
 Up to a constant depending on $q$ only, the same estimate as \eqref{Aug151} holds for $q$ in place of $1$ by virtue of the John-Nirenberg Proposition \ref{JNp}. This, coupled with \eqref{Aug150}, yields the required estimate \eqref{weakL1} and completes the proof of the weak-$L^1$ case. 
 
The proof of \eqref{Aug151} begins now. The crucial point   is the estimate
\begin{equation}  \label{Aug152} \|P_Q\cic{1}_{\Gamma_{\infty,\alpha,\beta}(y)}\|_{S ({\upsilon},\X)} \lesssim \lambda \beta^{-\frac d2} \left(1+ \frac{|y-c_Q|}{{\ell(Q)}}\right)^{-d-1} \end{equation}
holding uniformly in ${Q \in \cic{Q}_{f,\lambda}}$ and $y\in \mathbb R^d$,
 from which \eqref{Aug151} follows rather immediately. Indeed, fix a tent $\mathsf T=\mathsf T_{\alpha,\beta}(x,s)$. It is easy to see that $ \|P_Q\cic{1}_{\Gamma_{s,\alpha,\beta}(y)}\|_{S ({\upsilon},\X)} $ is nonzero for some $y\in B_{s}(x)$ only if $Q\subset B_{3s}(x)$. Hence
\[
\begin{split} & \quad 
\int_{B_s(x)} \mathcal \|P\cic{1}_{\Gamma_{s,\alpha,\beta}(y)}\|_{S ({\upsilon},\X)}  \,\d z \leq \sum_{   Q\subset B_{3s}(x)}  \int_{\R^d} \|P_Q\cic{1}_{\Gamma_{s,\alpha,\beta}(y)}\|_{S ({\upsilon},\X)}\,\d y \\ & \lesssim \lambda \beta^{-\frac d2}    \sum_{   Q\subset B_{3s}(x)} \int_{\R^d}  \left(1+{\textstyle \frac{|y-c_Q|}{{\ell(Q)}}}\right)^{-d-1}\ \d y  \lesssim \lambda   \sum_{   Q\subset B_{3s}(x)} |Q| \lesssim \lambda s^d \end{split}
\] 
where the last step follows from disjointness of ${Q \in \cic{Q}_{f,\lambda}}$. The above display is exactly the sought \eqref{Aug151}.

We still need to prove   estimate \eqref{Aug152} for each fixed $y\in \R^d$.  The point is that the support of  $P_Q(u,t)$ is contained in the complement of $\mathsf T_{\alpha,\beta}(c_Q,3\beta^{-1}{\ell(Q)})$. Hence,
if $|y-c_Q|\leq 3\ell(Q)$, then $\Gamma_{\infty,\alpha,\beta}(y)$ intersects such support only if  $  t\geq    \ell(Q)$. In fact, since $b_Q$ is supported on $Q$, 
\[
P_Q(u,t) \cic{1}_{\Gamma_{\infty,\alpha,\beta}(y)}\neq 0 \implies t\geq \max\{\ell(Q),|y-c_Q|\}.
\]  
Now since $b_j $ has mean zero and $\|b_Q\|_{L^1(\mathcal X)} \lesssim \lambda\ell(Q)^d$, we have 
\[
\begin{split}  
\left\|P_Q(u,t)\right\|_{\X}= \left\|  \int_{j} b_Q(z) [\phi_{  t}(u -z)- \phi_{ t}(u-c_Q)] \, \d y\right\|_{\X}   \lesssim \lambda \ell(Q)^{d+1}  t^{-(d+1)}
\end{split}
\]
where in the last step we have used that $\|\nabla \phi_{ t} \|_\infty\lesssim  t^{-(d+1)} $.  
Thus, squaring and  integrating first over $\{u\in \R^d: (u,t) \in \Gamma_{\infty,\alpha,\beta}(y)\}$, which has measure $\lesssim \beta^{-d}t^d$, 
\[\begin{split} &\quad \left\|P_Q\cic{1}_{\Gamma_{\infty,\alpha,\beta}(y)}\right\|_{S ({\upsilon},\X)}\leq \left(\int_{\Gamma_{\infty,\alpha,\beta}(y)}\ 
\|P_Q(u,t) \|_{\X}^2 \frac{\d u \d t}{t^{d+1}} \right)^{\frac12} \\ &
 \lesssim  \lambda \beta^{-\frac d2}  {\ell(Q)}^{(d+1)}\left(\int_{  t\geq   \max\{{\ell(Q)}, |y-c_Q|\}}    t^{-2(d+1)-1} \,   \d   t\right)^{\frac12}  \lesssim   \lambda \beta^{-\frac d2}   \left(1+ \frac{|y-c_Q|}{{\ell(Q)}}\right)^{-(d+1)}, 
\end{split}
\]
which is the claimed estimate \eqref{Aug152}. This completes the proof of \eqref{Aug151} and in turn of the weak-$L^1$ estimate (\ref{L1Sqemb1}).

\section{Overlapping Carleson embeddings and the RMF}\label{SecRMF}
  In this section, we detail the Carleson embedding theorems involving the \emph{overlapping} outer $L^p$ spaces of Section \ref{SecOM}. Unlike the lacunary case, here we work with a triple of complex Banach spaces $\X_1,\X_2,\X_3$ coupled with the identification \eqref{trcontfm}. We write $L^p(\mathsf{r}_q)$ for   $L^p(\R^{d+1}_+,\mathsf{r}_q^{\X_3})$.
\begin{proposition}[Overlapping Carleson Embedding Theorems] \label{ovcarlemb}
Assume that ${\X_3}$ has the nontangential   $\mathrm{RMF}$ property with respect to the above mentioned identification. 
Let $\phi:\R^d \to \R$ be a  member of the class $\cic{\Phi}$ of \eqref{decay}. For $f \in \mathcal{S}(\R^d; \X)$, $\alpha \in B_1(0),  0<\beta\leq 1$, define 
$$
F_{\phi,\alpha,\beta}(f)(u,t)= f* \phi_{\beta t} (u+\alpha t).$$ We then have
\begin{align}
\label{embRq} &\|F_{\phi,\alpha,\beta} (f)\|_{L^\infty(\mathsf{r}_q)} \lesssim\|f\|_{\mathrm{BMO}(\R^d; {\X_3})}, \qquad 0<q<\infty \\
& \|F_{\phi,\alpha,\beta} (f)\|_{L^{p}(\mathsf{r}_\infty)} \lesssim   \|f\|_{L^p(\R^d; {\X_3})},\qquad 1<p<\infty,\label{LpRqemb}
\\ &\|F_{\phi,\alpha,\beta} (f)\|_{L^{1,\infty}(\mathsf{r}_\infty)} \lesssim   \|f\|_{L^1(\R^d; {\X_3})}. \label{L1Rqemb}
\end{align}
with implicit constant depending on $q$ or $p$ and  only, as well as on the $\mathrm{RMF}$  character of $\X_3$.
\end{proposition} 
Again, a byproduct of the proof of inequality \eqref{L1Rqemb} is \begin{proposition}[Overlapping Carleson embedding theorem, local version]\label{cetlocover} In the same setting of the previous proposition, let  $f\in 
\mathcal S(  \R^d; \X_3) $  and  $  \cic{Q}_{f,\lambda} $ be as in \eqref{maxfctbd}.
  Then for all $q \in (0,\infty)$
$$\|F_{\phi,\alpha,\beta} (f)\|_{L^{\infty}(\R^{d+1}_+\backslash  \mathsf T_\beta(\cic{Q}_{f,\lambda}),\mathsf{r}_q^ {\X_3})} \leq C_{q} \lambda.$$ 
\end{proposition}
We pass to the proof of Proposition \ref{ovcarlemb}. Before the proof of each statement, we recall that the  nontangential   $\mathrm{RMF}$ property for $\X_3$ implies that the grand maximal function $\M_{\cic{\Phi}, {2/\beta}}$ is (suitably) bounded on each space in the right hand sides of \eqref{embRq}-\eqref{L1Rqemb}, with constant independent of the aperture $\beta$. Below, we restrict ourselves to the proof of the non-tilted case $\alpha=0,\beta =1$ and write  $F_{\phi}(f)$ instead of $F_{\phi,0,1,}$. The general case can be obtained by a tilting argument similar to Subsection \ref{sstilt}, and by further relying on the above $\beta$-independence.

\begin{proof}[Proof of \eqref{embRq}] By standard partition of unity arguments it suffices to prove the case of $\phi$ being supported in, say,  $B_\rho(0)$.
Let $\mathsf T=\mathsf T(x_0,s)$ be any fixed tent. Let $Q$ be  the cube centered at $x_0$ and with sidelength $6\rho s$. Let $g=f\cic{1}_Q$. We can assume $g$ has mean zero.    We have by definition of the size and by the support of $\phi$, and later by the definition of $\M_{\cic{\Phi}}$, the $L^q(\X_3)$-boundedness of $\M_{\cic{\Phi}}$ and John-Nirenberg's inequality, \begin{align*}
\left(\mathsf r_q(F_\phi f) (\mathsf T)\right)^q &= \frac{1}{s^d} \int_{B_s(x_0)}\big( \mathcal R\left(\big\{  f* \phi_t (y) : (y,t) \in \Gamma_s(x)  \big\} \right) \big)^q \,\d x \\ & \leq \frac{1}{s^d} \int_{B_s(x_0)}\big( \mathcal R\left(\big\{ (f\cic{1}_{B_{3\rho s}(x_0)})* \phi_t (y) : (y,t) \in \Gamma_s(x)  \big\} \right) \big)^q \,\d x \\ & \leq \frac{1}{s^d} \int_{B_s(x_0)}\big(\M_{\cic{\Phi}}(g) (x)\big)^q\,\d x \lesssim  s^{-d}\|g\|^q_{q} \lesssim  \|f  \|^q_{\mathrm{BMO}} 
\end{align*}
which is what was to be proved. 
\end{proof}
\begin{proof}[Proof of \eqref{LpRqemb}-\eqref{L1Rqemb}] We prove that
\begin{equation}
\label{wtb}
 \|F_\phi (f)\|_{L^{p,\infty}(\mathsf r_\infty) } \lesssim    \|f\|_{L^p(\R^d; {\X_3})}
\end{equation}
for all $1\leq p<\infty$, which includes \eqref{L1Rqemb}. Then \eqref{LpRqemb} follows by outer measure Marcin\-kie\-wicz interpolation. By   vertical scaling, we  can work with $\|f\|_p=1$. For each $\lambda>0$, we define  $$
E_\lambda= \bigcup_{I \in \cic{I}_\lambda}\mathsf  T(9I),  
$$
where $\cic{I}_\lambda$ is the collection of maximal dyadic cubes contained in  $$\{x \in \R^d:\M_{\cic{\Phi}}(f) (x) > \lambda\}.$$ Then $
\|E_\lambda\| \lesssim \lambda^{-p},$  where $\|E_\lambda\|$ denotes outer measure. To prove  \eqref{wtb} we are left to show that 
\begin{equation}
\label{bdt}
\sup_{(x_0,s) \in \R^{d+1}_+} \mathsf r_\infty\big(F_\phi f\cic{1}_{\R^{d+1}_+\backslash E_\lambda}\big) (\mathsf T(x_0,s)) \leq \lambda.
\end{equation}
The geometry of tents implies that, if $\Gamma_{s}(x) \not \subset  E_\lambda  $, one may find $2^d$ points $y_j \not \in \{\M_{\cic{\Phi}}(f) > \lambda\} $ such that  the truncated cone $\Gamma_{s}(x)$ is  covered by the union of   the cones $ \Gamma(y)$. Therefore, relying on the monotonicity property \eqref{therest}, 
$$
\big\|F_\phi f\cic{1}_{\Gamma_{s}(x) \cap (E_\lambda)^c}\big\|_{R({\upsilon},{\X_3})}\leq2^{d+1} \sup_{j=1,\ldots, 2^d}  \M_{\cic{\Phi}}(f) (y_j)    \lesssim  \lambda  
,$$
and taking supremum over $x,s$ yields \eqref{bdt}. The proof of \eqref{wtb} is complete.
  \end{proof}

\section{Proof of Proposition \ref{thmmodel}}\label{Secpfthmmodel} Throughout the  proofs that follow, the implied constants depend without explicit mention on $ (p_1,p_2)$,  on $K_\sigma$ and (polynomially) $\{\beta_j:j=1,2,3\}$, as well as  on the $\mathrm{UMD}$ and $\mathrm{RMF}$ character of the spaces.\begin{proof}[Proof of the strong-type estimate]    We can assume that $$\|f_1\|_{L^{p_1}( \R^d;\X_1)} =\|f_2\|_{L^{p_2}( \R^d; \X_2)}=\|f_3\|_{L^{p_3}( \R^d;\X_3)}=1$$ by linearity. For the sake of definiteness, we work in the case of $\pi=3$ and  $\sigma$ being the identity permutation: the other cases are identical. According to H\"older's inequality for outer measures Lemma \ref{outerholder}, we can take advantage of the $\mathcal R$-bound \eqref{ppc} on $\mathsf{w}$ to obtain
\[\begin{split} &\quad 
|{\mathsf{PP}_{\mathsf w}^3}(f_1,f_2,f_3)|\\ &\lesssim \left(\prod_{j=1,2}\|F_{\phi^j,\alpha_j,\beta_j}(f_j)\|_{L^{p_j}(\R^{d+1}_+,\mathsf{s}_2^{\X_j})}\right)\|F_{\phi^3,\alpha_3,\beta_3}(f_3)\|_{L^{p_3}(\R^{d+1}_+,\mathsf{r}_\infty^{\X_3})}\lesssim 1,\end{split}
\]
where the second inequality follows from Carleson embeddings  (\ref{LpSqemb}), (\ref{LpSqbelow}) of Proposition \ref{cet}, and (\ref{LpRqemb}) of Proposition \ref{ovcarlemb}.
\end{proof}
\begin{proof}[Proof of the weak-type estimate] {{} For simplicity of notation, we work in the case $\alpha_j=0,\beta_j=1$ for $j=1,2,3$ (and thus omit the subscripts $\alpha_j,\beta_j$ in the embedding maps $F_\phi$). The small modifications needed for the general case are mentioned at the end of the proof.}
By multilinear interpolation and symmetry, it suffices to prove the extremal case $p_1=p_2=1$. Observe that the assumption on $\mathsf{w}$ \eqref{ppc} is invariant with respect to horizontal scaling. Using  this  scaling and linearity, we can reduce to the case $$\|f_1\|_{L^{1}(\R^d;\X_1)} =\|f_2\|_{L^{1}(\R^d;\X_2)}=|F_3|=1.$$ 
We begin by setting \begin{equation}
\label{excsetdef}
 E=
\big\{x \in \R^d:\max_{j=1,2} \mathrm{M}( \|f_{j}\|_{\X_j})(x) >C \big\}.
\end{equation}
Let   $\cic Q$ be the collection of maximal dyadic cubes $Q$ such that $3Q$ is contained in $
E
$. 
Then we may choose $C>0$ large enough in \eqref{excsetdef} so that
 \begin{equation}
\label{excsetdef2} F_3':= F_3 \backslash \cup_{Q \in \cic Q} 9Q 
\end{equation}
 is a major subset of $F_3$.  
We now recall, for $j=1,2$ and $k\geq 0$, the definitions of collections $ \cic{Q}_{f_j, C2^{k}}$ from \eqref{maxfctbd}.
The key observation is that \begin{equation} \label{keyobs}\dist(F_3',3Q) \geq 2^{k}\ell(Q) \qquad \forall Q \in \cic{Q}_{f_j, C2^{k}}, \quad j=1,2, \quad  k\geq 0.
\end{equation}
Accordingly, we decompose
$$
F_{\phi^3 }(f_3)= \sum_{k=0}^\infty 2^{-40k} F_{\phi_k^3}(f_3)$$
where $\phi_k^3$ is supported on a cube of sidelength $2^{k}$ about the origin and has mean zero if $\phi^3$ does.  This support consideration and \eqref{keyobs} imply that    for all      subindicator $\X_3$-valued functions $f_3$ supported on $F_3'$,
\begin{equation} \label{suppF}
 \widehat{E_k}\cap \mathrm{supp}\, 
 F_{\phi_k^3 }(f_3) =\emptyset, \qquad \widehat{E_k}:= \mathsf{T}_1(\cic{Q}_{f_1, C2^{k}}) \cup  \mathsf{T}_1(\cic{Q}_{f_2, C2^{k}}),
\end{equation}
{{}referring to the notation \eqref{exccoll}}.
At this point, for the sake of definiteness, we work in the case of $\pi=3$ and of  $\sigma$ being the identity permutation: the other cases follow the same scheme,  replacing Propositions \ref{cet} and \ref{cetloc} by Propositions \ref{ovcarlemb} and \ref{cetlocover}, or viceversa, when appropriate.
Using  the Carleson embeddings of Proposition \ref{cetloc}, and  Proposition \ref{cet} estimate \eqref{L1Sqemb}, respectively,
$$
\|F_{\phi^j} (f_j)\|_{L^{\infty}(\R^{d+1}_+\backslash \mathsf{T}(\cic{Q}_{f_j, C2^{k}}), \mathsf{s}_2^{\X_j})} \lesssim 2^{k}, \qquad  \|F_{\phi^j }(f_j)\|_{L^{1,\infty}(\R^{d+1}_+ , \mathsf{s}_2^{\X_j})} \lesssim 1
$$ for $j=1,2$,
and logarithmic convexity yields
\begin{equation}
\label{excset2}
\|F_{\phi^j }(f_j)\|_{L^{2}(\R^{d+1}_+\backslash  \widehat{E_k}, \mathsf{s}_2^{\X_j})} \lesssim  2^k, \qquad j=1,2.\end{equation}
Finally, an application of Proposition \ref{ovcarlemb} ensures that 
\begin{equation}
\label{excset3}
\|F_{\phi^3_k }(f_3)\|_{L^{\infty}(\R^{d+1}_+ , \mathsf{r}_\infty^{\X_3})} \lesssim 1.\end{equation}
Relying on \eqref{suppF}, applying H\"older's inequality for outer measures Lemma \ref{outerholder}, and taking advantage of  \eqref{excset2} and \eqref{excset3}, we complete the proof as follows:
\[
\begin{split}
&\quad |{\mathsf{PP}_\mathsf{w}^3}(f_1,f_2,f_3)|\\   &= \sum_{k \geq 0}2^{-40k}\int_{\R^{d+1}_+\backslash  \widehat{E_k}} \big|\mathsf{w} \big( F_{\phi^1  }(f_1) , F_{\phi^2 }(f_2) ,   F_{\phi^3_k}(f_3)  \big)\big| \, \frac{\d x \d t}{t}\\ & \lesssim  \sum_{k \geq 0}2^{-40k} \left( \prod_{j=1,2}\|F_{\phi^j } (f_j)\|_{L^{2}(\R^{d+1}_+\backslash E_k , \mathsf{s}_2^{\X_j})} \right)  \|F_{\phi^3_k } (f_3)\|_{L^{\infty}(\R^{d+1}_+ , \mathsf{r}_\infty^{\X_3})}\\ & \lesssim 1, \end{split}
\]
which completes the proof.
{{} When the parameters $\alpha_j,\beta_j$ are generic, the same exact argument may be performed, with the following two modifications. In the definition \eqref{excsetdef2} of the set $F_3'$, replace $9$ by $9\beta^{-1}$, where $\beta=\min\{\beta_1,\beta_2\}$ and choose $C$ in \eqref{excsetdef} accordingly. In the definition \eqref{suppF}, still referring to the notation \eqref{exccoll}, replace $\mathsf T_1(\cic{Q}_{f_j, C2^{k}})$ by $\mathsf T_\beta(\cic{Q}_{f_1, C2^{k}})$ for $j=1,2$. }
    \end{proof}

%

\section{Proofs of the RMF lemmata}\label{ssprfrmf}
\begin{proof}[Proof of Lemma \ref{lemmastep}]  By dilation invariance, it suffices to argue in the case $\alpha=1$.
 Let $\phi \in \cic{\Phi}$. We recall from \cite[3.1.2, Lemma 2]{Bible} the decomposition
\begin{equation}\label{Rbdlemma1}
\phi= \sum_{k =0}^\infty \Psi_{2^{-k}} * \eta_k, \qquad \|\eta_{k}\|_1 \lesssim  2^{-2kd}. 
\end{equation}
By virtue of the compact support of $\Psi$, we have, for all $h \in L^1(\R)$, 
$$
  f*\Psi_t*h(y) = \int_{\R^d} f*\Psi_t(y-u) h(u) \, \d u = \int_{ \Gamma^{3\rho}(x)} f*\Psi_t(y-u) h(u) \,\d u
$$
thus, since $L^1$-averaging preserves $\mathcal R$-bounds (see Lemma \ref{Rbounds}), 
$$
\mathcal R\left(\big\{ f* \Psi_t*h (y) : (y,t) \in \Gamma(x)  \big\} \right) \leq \|h\|_1 \mathcal R\left(\big\{ f* \Psi_t (y) : (y,t) \in \Gamma^{3\rho}(x)  \big\} \right).
$$
Using additivity of $\mathcal R$-bounds, and coupling the   decomposition  \eqref{Rbdlemma1} with the above estimate written for $2^{-k}t$ in place of $t$ and $\eta_k$ in place of $h$, we get, for all $x \in \R^d$,
$$
\M_{\cic{\Phi}}  f (x)  \lesssim \sum_{k =0}^\infty 2^{-2dk} \mathcal R\left(\big\{ f* \Psi_{2^{-k}t} (y) : (y,t) \in \Gamma^{3\rho}(x)  \big\} \right)
= \sum_{k =0}^\infty 2^{-2dk} \M_{\Psi,3\rho2^k} f(x) .
$$
The proof is completed by taking $L^p$-norms and noting that    $$\|\M_{\Psi,3\rho2^k}  \|_{p\to p } = 2^{-\frac{dk}{p}} \|\M_{\Psi,3\rho }  \|_{p\to p } ,$$ by virtue of dilation invariance.
\end{proof}

\begin{proof}[Proof of Lemma \ref{lemmaRMF1}] The first statement, on $L^q$ bounds, follows by Lemma \ref{lemmastep} and interpolation of the  third (BMO bound) and fourth (weak-$L^1$ bound) statement. The fourth statement follows from the second ($H^1\to L^1$ bound) by a standard Calder\'on-Zygmund decomposition argument.

 By a variant of the argument in Lemma \ref{lemmastep} and a tensor product argument it suffices to check the second and third statement for $d=1$ and with $\M_{\cic{\Phi}}$ replaced by $\M_\Psi$  for a fixed choice of   nonnegative Schwartz function $\Psi:\R\to \R$ with integral 1 and support contained in $B_\rho(0)$.   

We begin with the proof of the $H^{1}(\R;{\X_3}) \to L^{1}(\R) $ bound. Let $f:\R^d \to {\X_3}$ be a function supported on an interval $Q\subset \R$, having mean zero and such that $\|f\|_{L^\infty(\R; 
\mathcal X_3)}\leq |Q|^{-1}$. Let $F$ denote the (compactly supported) primitive of $f$. Clearly $\|F\|_{p}\leq |Q|^{1/p}$.
We first bound 
\begin{equation}
\label{atom1}
\|\M_{\Psi} f \cic{1}_{3\rho Q}\|_{1} \lesssim  |Q|^{\frac{1}{p'}}  \|\M_{\Psi} f \|_{p} \lesssim |Q|^{\frac{1}{p'}}  \|  f \ \|_{p}  \lesssim |Q|  \|  f  \|_{\infty} \lesssim 1, 
\end{equation}
using the nontangential RMF$_p$ property in the second step. Let $$A_k=2^{k+1}{3\rho Q}\backslash 2^{k+1}{3\rho Q}.$$ Notice   that for $x \in A_k$, $f*\Psi_t(y)\cic{1}_{\Gamma(x)}(y,t)\neq 0$ if and only if  $t>2^{k}|Q|$. Again for  $x \in A_k$, integrating by parts, 
\[\begin{split}
\M_{\Psi} f (x)=\mathcal R\left(\big\{f*{\Psi_t} (y)\cic{1}_{\Gamma(x)}(y,t)  \big\} \right) &= \mathcal R\left(\big\{ t^{-1}F*{\Psi'_t} (y)  \cic{1}_{\Gamma(x)}(y,t)  \big\} \right) \\ &\lesssim  2^{-k} |Q|^{-1}  \M_{\cic{\Phi}} F (x)\end{split}
\]
using the obvious membership of $\Psi'$ to the class $\cic{\Phi}$.
By Lemma \ref{lemmastep}, we can employ the $L^p$ boundedness of  $\M_{\cic{\Phi}} $ as well, and obtain
\begin{align*}
\|\M_{\Psi} f \cic{1}_{(3\rho Q)^c}\|_{1}& \lesssim  |Q|^{-1} \sum_{k \geq 1} 2^{-k}   \|\M_{\Psi'} F \cic{1}_{A_k} \|_{1} \lesssim |Q|^{-1}  \sum_{k \geq 1} 2^{-k} \| \cic{1}_{A_k} \|_{p'} \|\M_{\cic{\Phi}} F \cic{1}_{A_k} \|_{p} \\ & \lesssim |Q|^{-\frac1p} \|F   \|_{p} \sum_{k \geq 1} 2^{-k/p} \lesssim 1.\end{align*}
Combining the last display with \eqref{atom1}, we conclude $\|\M_{\Psi} f  \|_{1} \lesssim 1$, and the  $H^1\to L^1$ bound follows.

We turn to the proof of the $\mathrm{BMO}(\R;{\X_3}) \to \mathrm{BMO}(\R)$ bound for $\M_{\Psi}$. Let us be given $f \in \mathrm{BMO}(\R;{\X_3})$ and an interval $Q$ with center $x_Q$. Setting $g= f\cic{1}_{3\rho Q}$,  we can assume $g$ has integral zero. 
Define$$
T_Q=\{(y,t): 0<t<3\rho |Q|, |y-x_Q|<3\rho |Q|-t\}, \; 
c_Q= \mathcal R\left(\big\{f*{\Psi_t} (y): (y,t) \not \in T_Q \big\} \right).
$$
Using subadditivity of $\mathcal R$-bounds with respect to union and the compact support of  $\Psi$, we have, for $x \in Q$, that 
\begin{align*}
&\quad\mathcal R\left(\big\{f*{\Psi_t} (y) : (y,t) \in  \Gamma(x)\big\} \right) \leq \mathcal R\left(\big\{f*{\Psi_t} (y): (y,t) \in T_Q \cap \Gamma(x)\big\} \right) + c_Q \\ &= \mathcal R\left(\big\{g*{\Psi_t} (y): (y,t) \in T_Q \cap \Gamma(x)\big\} \right) + c_Q \leq \M_{\Psi} g(x) + c_Q.    
\end{align*}
Thus, relying on John-Nirenberg inequality in the last step,
\begin{align*}
  \big\|(\M_\Psi f- c_Q)\cic{1}_Q \big\|_{p}^p \leq  \big\|\M_{\Psi} g \cic{1}_Q \big\|_{p}^p \leq   \| g \|_{p}^p \lesssim |Q|\|f\|_{\mathrm {BMO}}^p.
\end{align*}
Dividing and taking supremum over $Q$, another application of John-Nirenberg finishes the proof.
\end{proof}
  \begin{proof}[Proof of Lemma \ref{lemmaRMF2}] Assertion (1) follows from Kahane's contraction principle. Assertion (2) is immediate: see Remark \ref{remtype2}. 
To prove (3), we note that $\X_1,\X_2$, being UMD, have finite cotype, so that, for all finite sequences $\cic{\xi}=\{\xi_1,\ldots,\xi_N\} \subset {\X_3}$ 
$$
\mathcal R_{\mathcal{X}_1,\X_2 } (\cic{\xi}) \sim \Big\| t\mapsto \sup_{j=1,\ldots, N} |\xi_j(t)|\Big\|_{{\X_3}}
$$
with implicit constant depending only on the cotype character of $\X_k$, $k=1,2$. Thus
$$
\M_{\Psi} f(x) \lesssim \|M_{\mathrm{ltc}} f(x) \|_{{\X_3}}, \qquad M_{\mathrm{ltc}} f(x)= \Big\{ u \mapsto \sup_{(y,t) \in \Gamma(x)} |f(\cdot,u)*\Psi_t(y)|\Big\}.
$$
It is a result of Rubio De Francia \cite[Theorem 3]{RDF86} that $M_{\mathrm{ltc}}$ maps $L^p(\R^d,{\X_3})$ into itself for all $1<p<\infty$ if (and only if) ${\X_3}$ is UMD. This completes the proof of (2).

 The proof of  (4) follows the same outline of the corresponding dyadic version (see Remark \ref{dyRMF} below), but the noncommutative Doob's maximal inequality by Junge \cite{Junge02} is replaced by a result due to Tao Mei  \cite{Mei07}.
Recall that we have to bound the nontangential RMF corresponding to the $\mathcal R$-bound
$$
\mathcal R(\{a_1,\ldots,a_N\}) = \sup_{\|\{\lambda_n\}\|_{\ell^2_N}=1} \left(\E\Big\|\sum_{n=1}^N r_n\lambda_n  a_n
   \Big\|_{{\X_3}}^2\right)^{\frac12}.
$$
We turn to this task.
First of all, by Lemma \ref{lemmaRMF1}, it suffices to prove that ${\X_3}=L^p(\mathcal A, \tau)$ has the nontangential $\mathrm{RMF}_p$ property.  Write $L^p(\mathcal A, \tau)=L^p(\mathcal A)$ and similarly for other exponents. 
By a standard approximation argument, it suffices to show that, for a fixed  $\Psi$ as in Lemma \ref{lemmastep}, for any integer $N$ and any choice of  $y_n \in \R^d, t_n>0$, $n=1,\ldots, N$,  the maximal operator
$$
\widetilde \M f(x) =   \mathcal R\left(\big\{f* \Psi_{t_n} (y_n): (y_n,t_n) \in \Gamma(x), n=1,\ldots, N  \big\} \right).
$$
maps   $    L^p(\mathbb R; L^p(\mathcal A))  $ into $L^p$ boundedly.
We denote $B=L^\infty(\R) \otimes \mathcal A$, which is also a von Neumann algebra, and  identify $ L^p(\mathbb R;L^p(\mathcal A)) \sim L^p(B)  $. Then, we re  recall from  \cite[Theorem 3.4 (ii)]{Mei07} that for $f \in L^p(B)   $  there exists a pair  $a,b \in  L^p(B) $ and contractions $\xi_1,\ldots \xi_n \in B $ such that for all $ (y_n,t_n) \in \Gamma(x)$ 
$$
 F_n(x):=
f* \Psi_{t_n} (y_n) = a(x) \xi_n(x) b(x), \qquad \|a\|_{L^{2p}(\B)}\|b\|_{L^{2p}(\B)}\lesssim_p \|f\|_{L^{p}(\B)}.
$$ The argument of \cite{Mei07} is presented for the Poisson kernel, but it is easily seen to hold \emph{verbatim} for any $\Psi$ as in Lemma \ref{lemmastep}.
Now,
fix $\{\lambda_j\} \in \ell^2_N$ of unit norm. We then have
$$
 \E \left\|
\sum_{n=1}^N r_n \lambda_n F_n(x)   \right\|_{L^{p}(\mathcal A)} \leq  \E\left\| \sum_{n=1}^N r_n \lambda_n \xi_n(x) b(x)   \right\|_{L^{2p}(\mathcal A)} \|a(x)\|_{L^{2p}(\A)}.
$$
Therefore, using that ${L^{2p}(\mathcal A)} $ has type 2, \begin{align*} 
\E\left\| \sum_{n=1}^N r_n \lambda_n \xi_n(x) b(x)   \right\|_{L^{2p}(\mathcal A)} \lesssim \left( \sum_{n=1}^N \left\|  \lambda_n \xi_n(x) b(x)   \right\|_{L^{2p}(\mathcal A)}^2 \right)^{\frac12} \lesssim \|b(x)\|_{L^{2p}(\mathcal A)},
\end{align*}
and combining the last two estimates yields
$$
\widetilde \M f(x)\lesssim\|a(x)\|_{L^{2p}(\mathcal A)}\|b(x)\|_{L^{2p}(\mathcal A)}.
$$
Integrating over $x$ and using  $\|a\|_{L^{2p}(\B)}\|b\|_{L^{2p}(\B)}\lesssim_p \|f\|_{L^{p}(\B)}
$ we obtain  the claimed $L^p$ bound for $\widetilde \M$. The proof is thus completed.
\end{proof}

\section{Proof of Theorems \ref{onepara} and \ref{bipara}}
\label{SecT1pf}
 Once the UMD-valued outer measure theory is in place, our proofs have the same general scheme as \cite{DoThiele15},  essentially being simple applications of the lacunary Carleson embedding theorems of Section \ref{SecLCET}.
For the sake of clarity, we first present the case where $\Lambda$ is the dual form  to an operator $T$, which is the extension to UMD-valued functions of an operator acting on scalar valued functions and then discuss the necessary modifications for the operator-valued case.
\subsection{The case of a scalar kernel} We deal with the case $\X_1=\X,$ $\X_2=\X'$. Let $T$ be an operator mapping Schwartz functions on $\R$ into tempered distributions.  The $\X$-valued extension of $T$ is then defined as
\[
Tf=\sum_{k=1}^N T(\phi_k) x_k, \qquad f=\sum_{k=1}^N \phi_k x_k, \; \phi_k \in \mathcal S(\R), x_k\in \X.
\] 
With the notations of \eqref{Adef}, we assume that
\begin{equation}
\sup_{z,w \in \R^2_+}
A(z,w)|\langle T(\phi_{z}),\phi_{w}\rangle|\leq1. \label{T1ass}
\end{equation}
and prove that, for  $1<p<\infty$
\begin{equation}
\label{T1scal}
\|T\|_{L^p(\R;\X)\rightarrow L^p(\R;\X)}\leq C_p,
\end{equation}
for some constant $C_p$ depending only on $\phi$, $p$ and on the UMD character of $\X$.
By density, it suffices to prove $$
\Big|\int_\R \l Tf(x),g(x)\r \, \d x \Big|\lesssim  \|f\|_{L^p(\R;\X)}   \|g\|_{L^{p'}(\R;\X')} $$
with $f$, $g$ Schwartz functions respectively taking values in a finite rank subspace of $\X,\X'$.  For such functions,  Calder\'on's reproducing formula continues to hold:
\begin{equation}
\label{CRF}
f(\cdot)=C\int_0^\infty \int_{\R} F_{\phi}(f)(x,s)\phi_{x,s}(\cdot)\,\d x \frac{\d s}{s},
\end{equation}
and similarly for $g$. Then the continuity assumption on $T$ implies that
\[
\langle T(f),g\rangle=C\int_0^\infty \int_0^\infty \int_{\R}\int_{\R}\langle F_\phi(f)(x,s),F_\phi(g)(y,t)\rangle\langle T(\phi_{x,s}),\phi_{y,t}\rangle\,\d x\d y\frac{\d s}{s}\frac{\d t}{t}.
\]
Now set
\[
r:=\max (s,t,|y-x|).
\]
We split the integration domain into several parts and estimate each of them separately. When $r>|y-x|$, by symmetry it suffices to look at the region $r=s\geq t$, which yields
\[
\begin{split}
&|\int_0^\infty\int_{\R}\int_0^s\int_{x-s}^{x+s} \langle F_\phi(f)(x,s),F_\phi(g)(y,t)\rangle\langle T(\phi_{x,s}),\phi_{y,t}\rangle \,\d y\frac{\d t}{t}\d x\frac{\d s}{s} |\\
\lesssim &\int_0^\infty\int_{\R}\int_0^s\int_{x-s}^{x+s} | \langle F_\phi(f)(x,s),F_\phi(g)(y,t)\rangle| \,\d y\d t\d x\frac{\d s}{s^3},
\end{split}
\]
which by change of variables $y-x=\alpha s$ and $t=\beta s$ is equal to
\[
\int_{0}^1\int_{-1}^1\int_0^\infty\int_{\R}|\langle F_{\phi}(f)(x,s),F_{\phi,\alpha,\beta}(g)(x,s)\rangle| \,\d x\frac{\d s}{s}\d \alpha\d\beta.
\]
According to  Proposition \ref{outerCS} it is easy to see that the above is bounded by
\[
\begin{split}
&\lesssim\int_0^1\int_{-1}^1 \|F_{\phi}(f)\|_{L^p(\R^2_+,\mathsf{s}_p^{\X})}\|F_{\phi,\alpha,\beta}(g)\|_{L^{p'}(\R^2_+,\mathsf{s}_{p'}^{\X'})}\,\d\alpha\d\beta\\
&\lesssim \int_0^1\int_{-1}^1 \beta^{-\eps}\|f\|_{L^p(\R;\X)}\|g\|_{L^{p'}(\R;\X')}\,\d\alpha\d\beta\lesssim \|f\|_{L^p(\R;\X)}\|g\|_{L^{p'}(\R;\X')},
\end{split}
\]
where the first inequality on the second line is due to the Carleson embedding estimate \eqref{LqSqemb}
Now we turn to the region where $r=|y-x|$, and by symmetry it suffices to estimate the part $r=y-x$ and $t\leq s$. One has
\[
\begin{split}
&\quad|\int_0^\infty\int_{\R}\int_0^r\int_0^s \langle F_\phi(f)(x,s),F_\phi(g)(x+r,t)\rangle\langle T(\phi_{x,s}),\phi_{x+r,t}\rangle \,\frac{\d t}{t}\frac{\d s}{s}\d x\d r|\\
\leq &\int_0^\infty\int_{\R}\int_0^r\int_0^s |\langle F_\phi(f)(x,s),F_\phi(g)(x+r,t)\rangle| \,\d t\frac{\d s}{s} \d x\frac{\d r}{r^2},
\end{split}
\]
which by change of variables $s=\gamma r$ and $t=\beta r$ is equal to
\[
\int_0^1\int_0^\gamma \int_0^\infty\int_{\R} |\langle F_{\phi,0,\gamma}(f)(x,r),F_{\phi,1,\beta}(g)(x,r)\rangle|\,\d x\frac{\d r}{r}\d\beta\frac{\d\gamma}{\gamma}.
\]
Similarly as in the previous case, applying Proposition \ref{outerCS}  together with Theorem \ref{cet} shows that the last display is bounded by
\[
\begin{split}
&\lesssim \int_0^1\int_0^{\gamma} \|F_{\phi,0,\gamma}(f)\|_{L^p(\R^2_+,\mathsf{s}_p^{\X})}\|F_{\phi,1,\beta}(g)\|_{L^{p'}(\R^2_+,\mathsf{s}_{p'}^{\X'})}\,\d\beta\frac{\d\gamma}{\gamma}\\
&\lesssim \int_0^1\int_0^{\gamma} \gamma^{-\eps}\beta^{-\eps}\|f\|_{L^p(\R;\X)}\|g\|_{L^{p'}(\R;\X')}\,\d\beta\frac{\d\gamma}{\gamma}\lesssim \|f\|_{L^p(\R;\X)}\|g\|_{L^{p'}(\R;\X')}.
\end{split}
\]
As all the other symmetric regions can be estimated analogously, the proof of estimate \eqref{T1scal} is complete.
\subsection{Extension to operator-valued singular integrals} 
 
The proof of Theorem \ref{onepara} in its full generality is very similar to the scalar valued case. 
Again, we argue with $f$, $g$ Schwartz functions respectively taking values in a finite rank subspace of $\X_1,\X_2$, and employ Calder\'on's reproducing formula to write
\[
\Lambda(f,g)=C\int_0^\infty\int_{\R}\int_0^\infty\int_{\R}\Lambda(F_{\phi}(f)(x,s)\phi_{x,s},F_{\phi}(g)(y,t)\phi_{y,t})\,\d y\frac{\d t}{t}\d x\frac{\d s}{s}.
\]

As before, let $r:=\max(s,t, |y-x|)$. Using symmetry, it thus suffices to discuss the region $r=s$ and $r=y-x\geq s\geq t$. The main observation is that \eqref{newwbp} implies that the (weakly measurable) families of $B(\X_1,\X_2)$ bilinear forms \begin{align*} &
\Gamma^{<,\alpha,\beta}_{(x,s)}(\xi,\eta)= A(x,s,x+\alpha s, \beta s)\Lambda( \phi_{x,s}\xi, \phi_{x+\alpha s,\beta s} \eta),  \\ 
&\Gamma^{>,\alpha,\beta}_{(x,r)}(\xi,\eta)= A(x,\gamma r,x+r, \beta r)\Lambda(  \phi_{x,\gamma r}\xi, \phi_{x+r,\beta r} \eta ), \qquad 
\end{align*}
are $\mathcal R$-bounded by 1 (in particular, uniformly in $\alpha,\beta$).
We only discuss the region $r=s$ here. One has
\[
\begin{split}
&|\int_0^\infty\int_{\R}\int_0^s\int_{x-s}^{x+s} \Lambda(F_{\phi}(f)(x,s)\phi_{x,s},F_{\phi}(g)(y,t)\phi_{y,t})\,\d y\frac{\d t}{t}\d x\frac{\d s}{s}|\\
\leq&\int_0^1\int_{-1}^1\int_0^\infty\int_{\R} |\Gamma^{>,\alpha,\beta}_{(x,s)}(F_{\phi}(f)(x,s),F_{\phi,\alpha,\beta}(g)(x,s) \rangle| \,\d x\frac{\d s}{s}\d\alpha\d\beta,
\end{split}
\]
where the last inequality follows from the change of variables $y=x+\alpha s$, $t=\beta s$, and the fact that in the region $r=s$, $A(x,s,x+\alpha s,\beta s)=s/\beta$. 
Therefore, applying Proposition \ref{outerCS} with functions $F_1=F_{\phi}(f)$, $F_2=F_{\phi,\alpha,\beta}(g)$, and using  $\mathcal R$-boundedness of $\{\Gamma_{(x,s)}^{<,\alpha,\beta}\}$,  one obtains that the above is
\[
\lesssim \int_{0}^1\int_{-1}^1 \|F_{\phi}(f)\|_{L^p(\R^2_+,\mathsf{s}_p^{\X_1})}\|F_{\phi,\alpha,\beta}(g)\|_{L^{p'}(\R^2_+,\mathsf{s}_{p'}^{\X_2})}\,\d\alpha\d\beta\lesssim \|f\|_{L^p(\R;\X_1)}\|g\|_{L^{p'}(\R;\X_2)},
\]
which concludes the proof.

 \subsection{Proof of Theorem \ref{bipara}}

We proceed with the proof of Theorem \ref{bipara}. For $a\in L^q(\R;\X_1)$, $b\in L^{q'}(\R;\X_2)$, $z_1=(x_1,s_1)\in \R^2_+, w_1=(y_1,t_1)\in \R^2_+$, define the forms
\[
Q^1_{z_1,w_1}(a,b):=A(z_1,w_1)\Lambda(\phi_{z_1}a,\phi_{w_1}b).
\]
By a limiting argument, we can assume that $\X_1$ is finite dimensional, and that $\mathcal{S}(\R)\otimes L^q(\R;\X_1)$ is dense in $L^p(\R;L^q(\R;\X_1))$, and same for $\X_2$. Thus, we can identify $\Lambda$ with a form mapping $$\Lambda:L^p(\R;L^q(\R;\X_1))\times L^{p'}(\R;L^{q'}(\R;\X_2))\to \mathbb{C}.$$ We then have by Theorem \ref{onepara}
\[
\|\Lambda\|_{B(L^p(\R;L^q(\R;\X_1)), L^{p'}(\R;L^{q'}(\R;\X_2)))}\lesssim \mathcal{R}_{L^q(\R;\X_1),L^{q'}(\R;\X_2)}(\{Q^1_{z_1,w_1}:z_1,w_1\in\R^2_+\}),
\]
which by Lemma \ref{propertyalpha} below is bounded by
\[
\begin{split}
&\lesssim \mathcal{R}_{\X_1,\X_2}(\{(\xi,\eta)\mapsto A(z_2,w_2)Q^1_{z_1,w_1}(\psi_{z_2}\xi,\psi_{w_2}\eta):z_1,z_2,w_1,w_2\in\R^2_+\})\\
&=\mathcal{R}_{\X_1,\X_2}(\{Q_{z_1,z_2,w_1,w_2}:z_1,z_2,w_1,w_2\in\R^2_+\}).
\end{split}
\]
Thus the proof is complete.
It remains to prove the following lemma, which takes advantage of the property ($\alpha$) of the spaces $\X_1,\X_2$.
\begin{lemma}\label{propertyalpha}
Let $\X_1,\X_2$ have property $(\alpha)$. Then, for a family of bilinear forms $$\Lambda_j:\mathcal{S}(\R)\otimes\X_1\times\mathcal{S}(\R)\otimes\X_2\rightarrow \mathbb{C}, \qquad j \in\mathcal{J},$$ there holds
\[
\begin{split}
&\mathcal{R}_{L^{p}(\R;\X_1),L^{p'}(\R,\X_2)}(\{\Lambda_j:j\in\mathcal{J}\})\\
\lesssim &\mathcal{R}_{\X_1,\X_2}(\{(\xi,\eta)\mapsto A(z,w)\Lambda_j(\phi_{z}\xi,\phi_{w}\eta):z,w\in\R^2_+,j\in\mathcal{J}\}).
\end{split}
\]
\end{lemma}

\begin{proof}
Let $\Lambda_1,\ldots,\Lambda_N$ be such that the LHS is bounded by $$2\|\cic{\Lambda}\|_{B(\mathrm{Rad}(L^{p}(\R;\X_1)), \mathrm{Rad}(L^{p'}(\R;\X_2)))},$$ where
\[
\cic{\Lambda}\left(\sum_{j=1}^N \eps_jf_j,\sum_{j=1}^N \eps_jg_j\right):=\sum_{j=1}^N \Lambda_j(f_j,g_j).
\]
Identifying $L^{p}(\R;\mathrm{Rad}(\X_1))$ with $\mathrm{Rad}(L^{p}(\R;\X_1))$ and the same for $\X_2$,  Theorem \ref{onepara}  yields 
\[
\begin{split}
&\|\cic{\Lambda}\|_{B(\mathrm{Rad}(L^{p}(\R;\X_1)), \mathrm{Rad}(L^{p'}(\R;\X_2)))}=\|\cic{\Lambda}\|_{B(L^{p}(\R;\mathrm{Rad}(\X_1)),L^{p'}(\R;\mathrm{Rad}(\X_2)))}\\
\lesssim &\mathcal{R}_{\mathrm{Rad}(\X_1),\mathrm{Rad}(\X_2)}\Big(\big\{(\sum_{j=1}^N \eps_j\xi_j,\sum_{j=1}^N\eps_j\eta_j)\mapsto \sum_{j=1}^N A(z_j,w_j)\Lambda_j(\phi_{z_j}\xi_j,\phi_{w_j}\eta_j)\\ &\qquad\qquad \qquad \quad  :z_j,w_j\in\R^2_+,j\in\mathcal{J}\big\}\Big)\\
\leq &\mathcal{R}_{\X_1,\X_2}(\{(\xi,\eta)\mapsto A(z,w)\Lambda_j(\phi_{z}\xi,\phi_{w}\eta),\,z,w\in\R^2_+,j\in\mathcal{J}\}),
\end{split}
\]
where the last step follows from property ($\alpha$) in the form of Lemma \ref{lemmaalpha}.
\end{proof}
\section{Decomposition into model paraproducts}\label{Secdec}
In this section, we show how to recover a Coifman-Meyer  multiplier of the type occurring in Theorem \ref{CMT} as an average of model paraproducts. The argument is standard, but due to the unusual setting, we include some of the details. 

Let $f_j\in \mathcal S(\R^d; \X_j)$, $j=1,2,3$. Let $\phi:\R^d \to \R$ be a smooth radial, nonzero function with compact frequency support which does not contain the origin. Using Calder\'on's reproducing formula as in 
\eqref{CRF}, we may write
\begin{equation}\label{freAK}
\begin{split}
&\quad  \Lambda_m(f_1,f_2,f_3)
\\ &= \int_{(\R^{d+1}_+)^3}  \Lambda_m(F_\phi(f_1)(x,s) \phi_{x,s},F_\phi(f_2)(y,t) \phi_{y,t},F_\phi(f_3)(z,u) \phi_{z,u})\, \frac{\d z\d u}{u}\frac{\d y\d t}{t}\frac{\d x \d s}{s}
\end{split}
\end{equation}
The integral above can  be split into  three summands in the regions $t<A^{-1} s$, $t>As$, $A^{-1}s <t< As $ for $A>1$ fixed, each of which  is an average of paraproducts of a given type. We deal explicitly with the first region, which will yield a paraproduct of type 2. By a suitable change of variables, this part of \eqref{freAK} turns into
\begin{equation} \label{freak11}
\begin{split}
   \int_{\R^{d+1}_+} \int_{\R^d \times (0,A^{-1})} \int_{\R^d \times (B^{-1},B) } \Lambda_m\big(&F_\phi(f_1)(x,s) \phi_{x,s},F_\phi(f_2)(x+\alpha s,\beta s) \phi_{x+\alpha s,\beta s}, \\ & F_\phi(f_3)(x+\gamma s,\delta s) \phi_{x+\gamma s,\delta s}\big)   \,   \frac{s^d \d \gamma \d \delta}{\delta} \frac{s^d \d \alpha \d \beta}{\beta}\frac{\d x \d s}{s} 
\end{split}
\end{equation}
where the restriction on the range of $\delta$ comes from frequency support considerations, and the constant $B$ depends on $A$ and the support of $\phi$.
Let $
\psi= c\int_0^{A^{-1}} \phi
$.
Defining the trilinear forms on $\X_1,\X_2,\X_3$
\begin{equation} \label{wdef}
 \mathsf{w}_{x,s}^{\alpha,\gamma,\delta}(v_1,v_2,v_3):=s^{2d}\Lambda_m(v_1 \phi_{x,s}, v_2 \psi_{x+\alpha s, A^{-1}s},v_3\phi_{x+\gamma s,\delta s}), 
\end{equation}
 performing the integration in $\beta$ and exchanging the order of the remaining integrals, we obtain that \eqref{freak11} is equal to
 \begin{equation*}\begin{split}
   &\quad \int_{(\R^d)^2}  \int_{B^{-1}}^{B }\left( \int_{\R_{+}^{d+1}} \mathsf{w}_{x,s}^{\alpha,\gamma,\delta}\left(F_{\phi,0,1}(f_1)(x,s)  ,F_{\psi,\alpha,1}(f_2)(x, s),F_{\phi,\gamma,\delta}(f_3)(x,  s) \right)   \frac{\d x \d s}{s} \right) \\ & \times \d \alpha   \d \gamma  \frac{\d \delta}{\delta}.\end{split}
\end{equation*}
Using the Coifman-Meyer type condition \eqref{derivatives}, and the definition of \eqref{wdef}, it is easy to see that, for all $\delta \in (B^{-1},B ), \alpha,\gamma \in \R^d$,
\begin{equation}
\label{}\mathcal R_{\X_{2}| \X_{1}, \X_{3}}\big( \{\mathsf{w}_{x,s}^{\alpha,\gamma,\delta} : (x,s)\in \R^{d+1}_+\}\big)\lesssim (1+|\alpha| +|\gamma|)^{-N}
\end{equation}
for a suitable large $N$. That is, assumption \eqref{ppc} of the paraproduct theorem is satisfied with sufficiently fast decay in $\alpha,\gamma$, thus allowing to transport the bounds of Proposition \ref{thmmodel} to those of Theorem \ref{CMT}. This completes our decomposition procedure.

\bibliographystyle{amsplain}      

\bibliography{IUMJ-7466-DiPlinio}{}

\end{document}